\documentclass[11pt, reqno]{amsart}
\usepackage{amsmath}
\usepackage{amsthm, color}
\usepackage{amsfonts}
\usepackage{amssymb}
\usepackage{mathrsfs}
\usepackage[margin=1in]{geometry}
\usepackage{setspace}
\usepackage{tikz}
\usetikzlibrary{matrix, arrows, quotes}
\usepackage{tikz-cd}
\usepackage{enumitem}
\usepackage[colorlinks]{hyperref}
\hypersetup{linkcolor=black,citecolor=blue,filecolor=black,urlcolor=black} 
\usepackage{pdflscape}
\usepackage{array}
\usepackage{adjustbox}
\usepackage{tabularx}
\usepackage{multirow}
\usepackage{multicol}
\usepackage{arydshln}
\usepackage{comment}
\usepackage{ifthen}

\newcommand{\abs}[1]{\left\vert #1 \right\vert}
\DeclareMathOperator{\aChow}{CH}             		
\DeclareMathOperator{\atomfree}{af}
\newcommand{\af}{{\atomfree}}
\DeclareMathOperator{\at}{at}					
\DeclareMathOperator{\augmented}{aug}
\newcommand{\aug}{{\augmented}}
\renewcommand{\bar}[1]{\overline{#1}}

\DeclareMathOperator{\Chow}{\underline{CH}}		
\DeclareMathOperator{\cl}{cl}
\DeclareMathOperator{\coat}{coat}
\DeclareMathOperator{\coker}{Coker}

\renewcommand{\emptyset}{\varnothing}
\renewcommand{\epsilon}{\varepsilon}
\renewcommand{\hat}[1]{\widehat{#1}}
\DeclareMathOperator{\HF}{\mathrm{HF}}
\DeclareMathOperator{\HS}{\mathrm{H}}

\DeclareMathOperator{\init}{in}						
\newcommand{\iso}{\cong}
\DeclareMathOperator{\Ker}{Ker}
\renewcommand{\ker}{\Ker}
\newcommand{\kk}{\mathbb{K}}					
\DeclareMathOperator{\LL}{\mathcal{L}}

\renewcommand{\phi}{\varphi}
\DeclareMathOperator{\Poin}{P}

\newcommand{\QQ}{\mathbb{Q}}
\DeclareMathOperator{\reg}{reg}					
\DeclareMathOperator{\rk}{rk}

\renewcommand{\setminus}{\smallsetminus}	
\newcommand{\tensor}{\otimes}
\renewcommand{\tilde}[1]{\widetilde{#1}}
\DeclareMathOperator{\Tor}{Tor}
\newcommand{\term}[1]{\textbf{\textsf{#1}}}
\newcommand{\ZZ}{\mathbb{Z}}

\newcommand{\lattice}[3][]{
	\begin{tikzpicture}[#1]          
	\newcommand*\points{#2}     
	\newcommand*\edges{#3}          
	\newcommand*\scale{0.015}          
	\foreach \x/\y/\z/\w in \points {
		\node (\w) at (\x, \y) {\z}; 
	}
	\foreach \x/\y in \edges { \draw (\x) -- (\y); }      
	\end{tikzpicture}
}

\newtheorem{thm}{Theorem}[section]
\newtheorem{lemma}[thm]{Lemma}
\newtheorem{prop}[thm]{Proposition}
\newtheorem{cor}[thm]{Corollary}

\newtheorem*{main-thm}{Main Theorem}

\theoremstyle{definition}
\newtheorem{defn}[thm]{Definition}
\newtheorem{example}[thm]{Example}
\newtheorem{rmk}[thm]{Remark}
\newtheorem{question}[thm]{Question}

\newtheorem*{notation}{Notation}

\numberwithin{equation}{section}
\numberwithin{table}{section}

\title[Chow rings of matroids]{Chow rings of matroids are Koszul}
\author[M. Mastroeni]{Matthew Mastroeni}
\address{Iowa State University, Department of Mathematics, Ames, IA, USA}
\email{mmastro@iastate.edu}
\author[J. McCullough]{Jason McCullough}
\address{Iowa State University, Department of Mathematics, Ames, IA, USA}
\email{jmccullo@iastate.edu}

\begin{document}

\subjclass[2020]{Primary: 16S37, 13E10, 05B35; Secondary: 13H10, 05E40}

\keywords{Koszul algebra, Chow ring, matroid, lattice, Gorenstein ring}

\begin{abstract}
Chow rings of matroids were instrumental in the resolution of the Heron-Rota-Welsh Conjecture by Adiprasito, Huh, and Katz and in the resolution of the Top-Heavy Conjecture by Braden, Huh, Matherne, Proudfoot, and Wang.   The Chow ring of a matroid is a commutative, graded, Artinian, Gorenstein algebra with linear and quadratic relations defined by the matroid.  Dotsenko conjectured that the Chow ring of any matroid is Koszul.  The purpose of this paper is to prove Dotsenko's conjecture.  We also show that the augmented Chow ring of a matroid is Koszul.  As a corollary, we show that the Chow rings and augmented Chow rings of matroids have rational Poincar\'{e} series.
\end{abstract}

\maketitle

\begin{spacing}{1.1}

\section{Introduction}

Let $M$ be a simple matroid on a finite set $E$.  The Chow ring of $M$ is the quotient ring
\[\Chow(M) := \underline{S}_M/(\underline{I}_M + \underline{J}_M),\]
where
\[\underline{S}_M = \QQ[x_F \mid F \text{ is a nonempty flat of } M],\]
and $\underline{I}_M$ and $\underline{J}_M$ are the ideals:
\begin{align*}
\underline{I}_M &= ({\textstyle \sum_{i \in F}} \, x_F \mid i \in E)\\[1 ex]
\underline{J}_M &= \left(x_Fx_G \mid F,G \text{ are incomparable, nonempty flats of } M\right).
\end{align*}
In the case of matroids associated to complex hyperplane arrangements, Chow rings of matroids first appeared in the work of de Concini and Procesi \cite{DP95a} as the cohomology rings of wonderful compactifications of arrangement complements.  Building on this work, Feichtner and Yuzvinsky \cite{FY04} later showed how to associate a smooth toric variety to any finite atomic lattice and some extra combinatorial data called a building set, and they computed a similar presentation for its Chow ring.  What is now called the Chow ring of a matroid is the Chow ring of its lattice of flats with respect to its maximal building set; see Subsection~\ref{SSachow} for more details. 

Chow rings of matroids and their augmented variants (see Section \ref{Saugmented}) have garnered significant attention in recent years due to the important role they have played in the proof of the longstanding Heron-Rota-Welsh Conjecture \cite{AHK18} and more recently in the proof of the Top-Heavy Conjecture \cite{BHMPW20b}.  The proofs of these purely combinatorial conjectures rely heavily on showing that the Chow ring of a matroid has a number of nice algebraic properties; it is shown in \cite{AHK18} that $\Chow(M)$ is a graded, Artinian $\QQ$-algebra with Poincar\'{e} duality and versions of the hard Lefschetz and Hodge-Riemann conditions.  In particular, $\Chow(M)$ is a Gorenstein algebra admitting a presentation by quadratic relations, and so, it is natural to ask whether it is also a Koszul algebra. 

In fact, recent work of Dotsenko \cite{Dot20} shows that the cohomology ring of $\mathcal{\overline{M}}_{0,n}$, the compactification of the moduli space of $n$ marked points on $\mathbb{P}_\mathbb{C}^1$, is Koszul; these rings are Chow rings of intersection lattices of braid arrangements (with respect to the minimal building set).  Based on this work, Dotsenko explicitly conjectures that $\Chow(M)$ is always Koszul  \cite[Conjecture 5.3]{Dot20}.  Our first main result is to confirm Dotsenko's conjecture.\footnote{There is potential for confusion on terminology here.  We follow \cite{AHK18} by defining the Chow ring of a matroid to be the Chow ring of the associated geometric lattice with respect to the maximal building set.  Our main Theorem~\ref{mainthm1} shows that all such Chow rings are Koszul.  As the cohomology ring of $\mathcal{\overline{M}}_{0,n}$ is a Chow ring with respect to the minimal building set, our theorem does not recover Dotsenko's result.  We explain this in greater detail in Subsection~\ref{SSachow}.  We also note that the more general statement that the Chow ring of a geometric lattice with respect to any building set is Koszul is false; see Example~\ref{nonKoszul:building:set}.  
}

\begin{thm}\label{mainthm1}
For any simple matroid $M$, $\Chow(M)$ is Koszul.
\end{thm}

Although it is well-known that all Koszul algebras have quadratic presentations (we will call such algebras quadratic),  not all quadratic algebras are Koszul.  In fact, it is not even guaranteed that a quadratic algebra will be Koszul if it is also Artinian and Gorenstein; the first non-Koszul such algebra was constructed by Matsuda \cite{Mat18}.  A thorough study of the Koszul properties of quadratic, Gorenstein algebras was conducted by Schenck, Stillman, and the first author in \cite{MSS21, MSS22}.  In particular, they determined (with only 3 exceptions) the ordered pairs $(r,c)$ for which there is a non-Koszul, quadratic, Gorenstein algebra of characteristic zero with regularity $r$ and codimension $c$; see also \cite{MS20}.  

In order for a quadratic algebra to be Koszul, it is sufficient for the defining ideal of relations to have a quadratic Gr\"obner basis, yet there are Koszul algebras whose defining ideals have no quadratic Gr\"obner bases with respect to any monomial order or any linear change of variables.  Previously, Feichtner and Yuzvinsky \cite{FY04} gave a non-quadratic Gr\"obner basis for a more general construction of the Chow ring of an atomic lattice (see Section~\ref{Sback}).  An argument similar to \cite[Proposition 3.1]{Dot20} shows that there is no possible quadratic Gr\"obner basis for $\underline{I}_M + \underline{J}_M$ in the given presentation. 

Our approach to showing that $\Chow(M)$ is always Koszul involves the construction of a Koszul filtration (see Theorem~\ref{Chow:rings:of:matroids:are:Koszul}).  A key ingredient in the construction is the study of lattices that admit what we call a total coatom ordering.  This is an ordering on the elements of the lattice satisfying certain properties similar to that of a shellable simplicial complex.  This aspect of the paper was inspired by and is similar to the notion of a strongly shellable simplicial complex of Guo, Shen, and Wu \cite{GSW19}.  A very similar notion is implicit in the work of Delucchi \cite{Del08}.  In Section~\ref{Smeetshellable}, we show that the lattice of flats of a matroid always admits a total coatom ordering thereby making it possible to always construct a Koszul filtration on $\Chow(M)$.  

It is also natural to study the Poincar\'{e} series of $\Chow(M)$.  The Poincar\'{e} series of a $\kk$-algebra $R$, denoted $\Poin_\kk^R(t)$, is the generating function of the total Betti numbers of $\kk$ over $R$:
\[\Poin_\kk^R(t) := \sum_{t \ge 0} \dim_\kk \Tor^R_i(\kk,\kk)t^i.\]
For many years, it was expected that all such Poincar\'{e} series were rational; this expectation became known as the Serre-Kaplansky Problem.  In \cite[Example 7.1]{Ani82}, Anick gave the first example of a commutative $\kk$-algebra with irrational Poincar\'{e} series.  Later examples of Artinian, Gorenstein algebras with irrational Poincar\'{e} series included one by B{\o}gvad \cite{Bog83}.  We observe that this algebra is quadratic (see Example~\ref{EXirrational}).  Thus, not all quadratic, Artinian, Gorenstein $\kk$-algebras have rational Poincar\'{e} series.  Our second main result, a corollary to the first, is the following.

\begin{thm}\label{mainthm2} 
For any simple matroid $M$, $\Chow(M)$ has a rational Poincar\'{e} series.
\end{thm}

We also consider the augmented Chow ring $\aChow(M)$ of a matroid $M$.  These algebras were introduced in \cite{BHMPW20a} as an analog of the Chow ring of the augmented wonderful variety associated to a hyperplane arrangement in the realizable case.  Like the Chow ring of a matroid, these are quadratic, Artinian, Gorenstein algebras satisfying the K\"ahler package.  We show (Theorem~\ref{augmented:Chow:rings:of:matroids:are:Koszul}) that the augmented Chow ring of a matroid is also Koszul.  

The rest of the paper is organized as follows.  Section~\ref{Sback} collects necessary definitions and background results on lattices, matroids, Chow rings, and Koszul algebras.  Section~\ref{Smeetshellable} contains a proof that all geometric lattices admit a total coatom ordering.  Section~\ref{Smain} contains the proofs of our main results for Chow rings of matroids, while Section~\ref{Saugmented} contains the analogous results for augmented Chow rings of matroids.  A final Section~\ref{Sexamples} collects some examples and questions.

\section{Background}\label{Sback}

In this section, we collect some necessary notation and background for our main results.  The reader can safely skip to Section~\ref{Smeetshellable} and review this section as necessary.

\begin{notation}
Throughout this section, $R = \oplus_{i \ge 0} R_i$ will denote a commutative graded $\kk$-algebra finitely generated as an algebra over $R_0 = \kk$ by elements of degree one.  We call such an algebra \term{standard graded} and set $R_+ = \oplus_{i \ge 1} R_i$ to be the graded maximal ideal.
\end{notation} 

\subsection{Hilbert functions and free resolutions}\label{SSkoszul}

 Let $M = \oplus_{i \in \ZZ} M_i$ be a finitely generated, graded $R$-module.  The \term{Hilbert function} of $M$ is defined as $\HF_M(i) := \dim_\kk M_i$, and its generating function is the \term{Hilbert series} of $M$, denoted $\HS_M(t) := \sum_{i \ge 0} \HF_M(i) t^i$.  The module $M$ has a \text{minimal graded free resolution} $\mathbf{F}_\bullet$ over $R$, which is an exact sequence of the form
\[\mathbf{F}_\bullet: \cdots \xrightarrow{\partial_{i+1}} F_i \xrightarrow{\partial_{i}} F_{i-1} \xrightarrow{\partial_{i-1}} \cdots \xrightarrow{\partial_2} F_1 \xrightarrow{\partial_1} F_0,\]
where $M \iso \coker(\partial_1)$, each $F_i = \oplus_{j} R(-j)^{\beta^R_{i,j}(M)}$ is a finite-rank, graded, free $R$-module, and all maps are graded homomorphisms.  Here, $R(-j)$ is the rank-one free $R$-module with graded components $R(-j)_i = R_{i-j}$.  The numbers $\beta^R_{i,j}(M)$ are the \term{graded Betti numbers} of $M$ over $R$.  The \term{total Betti numbers} are denoted 
\[ \beta_i^R(M) := \sum_j \beta_{i,j}^R(M),\]  
and their generating function
\[ \Poin_M^R(t) = \sum_{i \ge 0} \beta_i^R(M)t^i \]
is called the \term{Poincar\'{e} series} of $M$.  The \term{regularity} of $M$ is 
\[ \reg_R(M) = \max\{j-i \mid \beta_{i,j}^R(M) \neq 0 \text{ for some $i$}\}. \]  
When no maximum exists, we say that $M$ has infinite regularity.  When $M$ is Artinian and $R$ is a polynomial ring, $\reg_R(M) = \max\{i \mid M_i \neq 0\}$.  See \cite{Eis05} for connections between regularity, free resolutions, and local cohomology.

\subsection{Koszul algebras}\label{SSKoszul}
A standard graded $\kk$-algebra $R$ is called \term{Koszul} if $\kk$ has a linear free resolution over $R$; equivalently, $R$ is Koszul if and only if $\reg_R(\kk) = 0$.  This definition extends to the noncommutative case, but all rings we consider in this paper are commutative.  Koszul algebras are ubiquitous in algebraic geometry and topology; they appear as coordinate rings of canonical curves \cite{VF93}, quotients of polynomial rings by quadratic monomial ideals \cite{Fro75}, and all sufficiently high Veronese subalgebras of standard graded $\kk$-algebras \cite{Bac86}.  There are several equivalent conditions to being Koszul.  Notably, Fr\"oberg \cite{Fro99} showed that $R$ is Koszul if and only if 
\begin{align}\label{PHS1}\Poin_\kk^{\,R}(t)\HS_R(-t) &= 1.\end{align}
Moreover, it follows from work of Avramov, Eisenbud, and Peeva \cite{AE92, AP01} that $R$ is Koszul if and only if all finite, graded $R$-modules have finite regularity.  See \cite{Con14} and \cite{Fro99} for surveys of the theory of Koszul algebras.  

Any commutative Koszul algebra has a presentation $S/I$, where $S = \kk[x_1,\ldots,x_n]$ is a standard graded polynomial ring over $\kk$ and $I$ is an ideal generated by quadrics (homogeneous polynomials of degree $2$); we call such rings \term{quadratic}.  Not all quadratic $\kk$-algebras are Koszul.  It is sufficient for $I$ to have a quadratic Gr\"obner basis with respect to some monomial order, perhaps after a change of coordinates; such algebras are called \term{G-quadratic}. Yet not all Koszul algebras are G-quadratic; see \cite[Remark 1.14]{Con14}.  If $\ell \in R$ is a nonzerodivisor, then $R$ is Koszul if and only if $R/\ell R$ is Koszul; see e.g. \cite[Theorem 3.1]{CDR13}.  

When considering the Chow ring of a matroid in the introduction, there are two approaches one might take to prove the Koszul property.  Since quadratic monomial ideals, like $\underline{J}_M$, are Koszul, if the linear forms in $\underline{I}_M$ formed a regular sequence on $S_M/\underline{J}_M$, one could conclude Koszulness for $\Chow(M)$ immediately; however, in general, these linear forms do not form a regular sequence.  Another natural approach is to look for a quadratic Gr\"obner basis for $\underline{I}_M + \underline{J}_M$.  Yet, as Dotsenko has noted \cite{Dot20}, in the given presentation there are obstructions that prevent the existence of a quadratic Gr\"obner basis.  It is not clear if a change of basis would resolve this obstruction.  We address this issue in Section~\ref{Sexamples}.

In the absence of a quadratic Gr\"obner basis, one way to prove that a $\kk$-algebra has the Koszul property is to construct a Koszul filtration, a concept first formally defined by Conca, Trung, and Valla.  

\begin{defn} Let $R$ be a standard graded $\kk$-algebra.  A family $\mathcal{F}$ of ideals of $R$ is said to be a \term{Koszul filtration} of $R$ if the following conditions hold.
\begin{enumerate}
    \item Every ideal $I \in \mathcal{F}$ is generated by linear forms.
    \item The ideals $(0)$ and $R_+$ are in $\mathcal{F}$.
    \item For every nonzero $I \in \mathcal{F}$, there exists an ideal $J \in \mathcal{F}$ such that $J \subsetneq I$, $I/J$ is cyclic, and $(J : I) \in \mathcal{F}$.
\end{enumerate}
\end{defn}

If $R$ has a Koszul filtration, it follows from an inductive argument on the ideals in the filtration that $R$ is Koszul \cite[Proposition~1.2]{CTV01}.

\subsection{Lattices}\label{SSlattices}

Recall that a (finite) \term{lattice} is a finite partially ordered set $(\LL,\le)$ in which every two elements $a,b \in \LL$ have a unique least upper bound called the \term{join}, denoted $a \vee b$, and a unique greatest lower bound called the \term{meet}, denoted $a \wedge b$.  Consequently, all nonempty lattices have a unique minimum element and maximum element, denoted $\hat{0}$ and $\hat{1}$ respectively.

Given $a,b \in \LL$ with $a \leq b$, the \term{intervals} $[a,b]$ and $(a, b]$ are the sets: 
\begin{align*}
[a,b] &= \{c \in \LL \mid a \le c \leq b\} \\[1 ex] 
(a,b] &= \{c \in \LL \mid a < c \le b\}.
\end{align*}  
Clearly, every interval $[a, b]$ is also a lattice.  If $a < b$ and there are no elements $c \in \LL$ with $a < c < b$, then we say that $b$ \term{covers} $a$ or $b$ is a \term{cover} of $a$, and we write $a \lessdot b$.  The \term{Hasse diagram} of $\LL$ is a graph with vertices labeled by the elements of $\LL$ positioned in such a way that if $a < b$ then vertex $a$ is lower than vertex $b$; an edge from $a$ to $b$ is drawn exactly when $b$ covers $a$.  See Figure \ref{figure:trunc} for an example.    

The elements $a \in \LL$ that cover $\hat{0}$ are called the \term{atoms} of $\LL$, collectively denoted $\at(\LL)$.  The elements $b \in \LL$ covered by $\hat{1}$ are called the \term{coatoms} of $\LL$, collectively denoted $\coat(\LL)$.  For each element $\hat{0} \neq a \in \LL$, we also let $\coat(a)$ denote the set of elements of $\LL$ covered by $a$, which are precisely the coatoms of the lattice $[\hat{0}, a]$.

The lattice $\LL$ is \term{atomic} if every element of $\LL$ is a join of finitely many atoms.  We say that $\LL$ is \term{graded} (or \term{ranked}) if there is a function $\rk:\LL \to \ZZ_{\ge 0}$ such that $\rk \hat{0} = 0$ and if $a,b \in \LL$ and $a \lessdot b$, then $\rk b = \rk a + 1$, in which case the number $\rk \LL  = \max\{\rk a \mid a \in \LL\}$ is called the \term{rank} of $\LL$.  The lattice $\LL$ is \term{semimodular} if for all $a, b \in \LL$ such that $a \wedge b \lessdot a$ and $a \wedge b \lessdot b$, then $a \lessdot a \vee b$ and $b \lessdot a \vee b$; equivalently, $\LL$ is semimodular if it is graded and its rank function satisfies
\[\rk a + \rk b \ge \rk(a \wedge b) + \rk(a \vee b),\]
for all $a,b \in \LL$.  Finally, a lattice is called \term{geometric} if it is finite, atomic, and semimodular.  For basic results on lattices, we refer the reader to \cite{Sta12}.

\subsection{The lattice of flats of a matroid}\label{SSmatroids}

A \term{matroid} $M$ is a pair $(E,\mathcal{I})$ consisting of a finite set $E$, called the \term{ground set} of $M$, and a collection $\mathcal{I}$ of subsets of $E$ satisfying three properties:
\begin{enumerate}
    \item $\varnothing \in \mathcal{I}$.
    \item If $I \in \mathcal{I}$ and $I' \subseteq I$, then $I' \in \mathcal{I}$.
    \item If $I_1,I_2 \in \mathcal{I}$ and $|I_1| < |I_2|$, then there exists an element $e \in I_2 \smallsetminus I_1$ such that $I_1 \cup e \in \mathcal{I}$.
\end{enumerate}

\begin{notation}
In the above definition and the subsequent sections, we abuse notation and identify elements $e \in E$ with the corresponding set $\{e\}$.  Thus, if $F \subseteq E$, we will frequently write $F \cup e$ in place of $F \cup \{e\}$.
\end{notation}

The members of $\mathcal{I}$ are called \term{independent sets} of $M$.  A subset of $E$ that is not in $\mathcal{I}$ is called \term{dependent}.  A maximal independent set is called a \term{basis}.  All bases of a matroid have the same cardinality, called the \term{rank} of $M$.  Given a subset $X \subseteq E$, the \term{rank} of $X$, denoted $\rk_M X$, is the cardinality of the largest independent set contained in $X$; we drop the subscript when the matroid is clear from context.  

The \term{closure} of a subset $X \subset E$ in $M$ is 
\[ \cl(X) = \{e \in E \mid \rk(X \cup e) = \rk X \}.\]  
A subset $F$ is called a \term{flat} of $M$ if $F = \cl(F)$.  A \term{hyperplane} of $M$ is a flat $H$ of rank $\rk H = \rk M - 1$ is called.  One can give an equivalent definition of matroid as a pair $(E,\LL)$ consisting of a finite set $E$ and a collection $\LL$ of subsets of $E$, called flats, satisfying:
\begin{enumerate}
    \item If $F,G \in \LL$, then $F \cap G \in \LL$.
    \item If $F \in \LL$ and $e \in E \smallsetminus F$, then there is a unique flat $G \in \LL$ that minimally contains $F \cup e$.
\end{enumerate}
Matroids can also be characterized by their rank functions, by their bases, or by their minimal dependent sets.

The set of flats of $M$, denoted $\LL(M)$, has the structure of a lattice; for any two flats $F,G \in \LL(M)$, the meet is the intersection, $F \wedge G = F \cap G$, and the join is the closure of the union, $F \vee G = \cl(F \cup G)$.  The atoms of $\LL(M)$ are precisely the rank-one flats, and the coatoms of $\LL(M)$ are precisely the hyperplanes of $M$.  A lattice is geometric if and only if it is isomorphic to the lattice of flats of some matroid \cite[Theorem 1.7.5]{Oxl11}.  However, this correspondence is not quite bijective, as different matroids can have isomorphic lattices of flats. 

An element $e \in E$ is a \term{loop} of $M$ if the set $\{e\}$ is dependent.  
If $e,f \in E$ are not loops, then $e$ and $f$ are \term{parallel} if $\{e,f\}$ is dependent.  A matroid is \term{simple} if it has no loops and no pairs of parallel elements.  For any matroid $M$, there is a unique simple matroid (up to isomorphism) whose lattice of flats is isomorphic to $\LL(M)$, called the  \term{simplification} of a matroid $M$.  It can be constructed as the matroid on the set of rank-one flats of $M$ such that a set of flats $\{Y_1, \dots,Y_t\}$ is independent if and only if $\rk_M(Y_1 \vee \cdots \vee Y_t) = t$.  

\begin{rmk}
Since it is clear from the definition of the Chow ring of a matroid that any two matroids with isomorphic lattices of flats will have isomorphic Chow rings, we can always replace a matroid $M$ with its simplification, and so, we may henceforth assume that all matroids under consideration are simple without any loss of generality. 
\end{rmk}

There are two constructions for producing new matroids from a given matroid that will play an important role in the subsequent sections.  Given a matroid $M$ on a ground set $E$ and a subset $S \subseteq E$, the \term{restriction} of $M$ to $S$ is the matroid $M\vert S$ on the ground set $S$ whose flats are of the form $F \cap S$ for some flat $F$ of $M$.  In particular, when $H$ is a flat of $M$, every flat of $M\vert H$ is also a flat of $M$ so that $M \vert H$ is a matroid quotient of $M$, and the lattice of flats of $M\vert H$ is just the interval $[\emptyset, H]$ in $\LL(M)$.  The \term{truncation} of $M$ is the matroid $T(M)$ on the ground set $E$ with rank function $\rk_{T(M)} X = \max\{\rk_M X, \rk M - 1\}$.  The lattice of flats of the truncation $T(M)$ is obtained by removing all the hyperplanes from the lattice of flats of $M$.  We refer the reader to \cite{Oxl11} for further details about these constructions as necessary.

\subsection{Chow rings of atomic lattices}\label{SSachow}

We now recall the more general definition of the Chow ring of an atomic lattice from \cite{FY04}.  Let $\LL$ be an atomic lattice.  A subset $\mathcal{G} \subseteq \LL\smallsetminus\{\hat 0\}$ is a \term{building set} if for any $G \in \LL\smallsetminus \{\hat 0\}$, there is a poset isomorphism
\[\phi_G:\prod_{i = 1}^t \,[\hat0,G_i] \to [\hat0,G],\]
where $\max \mathcal{G}_{\leq G} =  \{G_1,\ldots,G_k\}$ and $\phi_G(\hat0,\ldots,\hat0,G_i,\hat0,\ldots,\hat0) = G_i$ for $1 \leq i \leq t$.  Here, $\max \mathcal{G}_{\leq G}$ denotes the set of maximal elements among all elements in $\mathcal{G}$ that are less than or equal to $G$.  Every atomic lattice $\LL$ has a unique maximal building set $\mathcal{G}_{\max} = \LL \smallsetminus \{\hat 0\}$ and a unique minimal building set $\mathcal{G}_{\min}$ consisting of the irreducible elements in $\LL \setminus \{\hat{0}\}$. (The atoms of $\LL$ are always irreducible, but in general, there can be more irreducibles than just the atoms; see \cite[Section~2]{incidence:combinatorics}.)  Given any building set $\mathcal{G}$ of $\LL$, a subset $\mathcal{S} \subseteq \mathcal{G}$ is called \term{nested} if for any pairwise incomparable elements $G_1,\ldots,G_t \in \mathcal{S}$, either $t \le 1$ or $G_1 \vee \cdots \vee G_t \notin \mathcal{G}$.  The set of nested subsets of $\mathcal{G}$ is an abstract simplicial complex denoted by $\mathcal{N}(\LL,\mathcal{G})$.  

\begin{defn} Let $\LL$ be a finite, atomic lattice and building set $\mathcal{G}$ for $\LL$.  Then the \term{Chow ring of $\LL$ with respect to $\mathcal{G}$} is the algebra
\[D(\LL,\mathcal{G}) := \mathbb{Z}[x_G\mid G \in \mathcal{G}]/\mathcal{I},\]
where $\mathcal{I}$ is generated by
\[ \prod_{i = 1}^t G_i \qquad \text{for} \qquad \{G_1,\ldots,G_t\} \notin \mathcal{N}(\LL,\mathcal{G})\]
and
\[\sum_{G \ge A} x_G \qquad \text{for} \qquad A \in \at(\LL).\]
\end{defn}
\noindent Note that $D(\LL,\mathcal{G})$ is a quotient of the Stanley-Reisner ring associated to the simplicial complex $\mathcal{N}(\LL,\mathcal{G})$. 

For any building set, Feichtner and Yuzvinsky showed that $D(\LL,\mathcal{G})$ is the Chow ring of a certain smooth, affine, toric variety \cite[Theorem~3]{FY04}.  They also show that for any essential complex hyperplane arrangement $\mathcal{A}$ and any building set $\mathcal{G}$ for the intersection lattice $\LL(\mathcal{A})$ containing $\{0\}$, the cohomology ring of the De Concini and Procesi wonderful compactification $\bar{Y}_{\mathcal{A}, \mathcal{G}}$ of the arrangement \cite{DP95b, DP95a} is isomorphic to $D(\LL(\mathcal{A}),\mathcal{G})$.  

At the other extreme, the Chow ring of a matroid $M$ is the special case corresponding to the maximal building set of the lattice of flats of $M$; that is,
\[\Chow(M) \iso D(\LL(M),\mathcal{G}_{\mathrm{max}}) \tensor_\mathbb{Z} \mathbb{Q}.\]
We stress the importance of which building set one chooses in the following remark. 

\begin{rmk}\label{chow:remark} While reading \cite{BES20} we discovered the following discrepancy:
It is claimed in \cite[Remark 3.2.2]{BES20} that the Chow ring of the graphic matroid $M(K_n)$ of a complete graph on $n$ vertices is the cohomology ring of the Deligne-Mumford space $\bar{\mathcal{M}_{0, n+1}}$, which was shown to be Koszul by Dotsenko.  This claim may be somewhat misleading.  Although Dotsenko proves the cohomology ring of $\bar{\mathcal{M}_{0, n+1}}$ is Koszul \cite{Dot20} and the cohomology ring is a Chow ring associated to the lattice of flats of $M(K_n)$, it is not \emph{the} Chow ring of $M(K_n)$.  The lattice of flats of $M(K_n)$ is isomorphic to the partition lattice $\Pi_n$ of partitions of $\{1, \dots, n\}$ ordered by refinement. The distinction between the two rings comes down to the choice of building set. The Chow ring of $M(K_n)$ is the Chow ring $D(\Pi_n, \mathcal{G}_{\max})$ with respect to the maximal building set $\mathcal{G}_{\max} = \Pi_n \setminus \{\hat{0}\}$.  On the other hand, as pointed out in \cite[\S 7]{FY04}, the cohomology ring of $\bar{\mathcal{M}_{0, n+1}}$ is isomorphic to $D(\Pi_n, \mathcal{G}_{\min})$ with respect to the minimal building set.  These two rings are not isomorphic, which can easily be seen for $n = 4$ from the fact that both rings are Artinian with nondegenerate presentations of different codimensions.  Thus, to the best of our knowledge, there is not any large class of matroids whose Chow rings were previously known to be Koszul.
\end{rmk}

\section{Lattices with total coatom orderings}\label{Smeetshellable}

As Bj\"orner writes in \cite[p. 232]{Bjo92}, \emph{``A number of remarkable properties of matroids are revealed by, but not dependent on, assigning a linear order to the underlying point set.''}
In this section, we follow Bj\"orner's lead and define the notion of a total coatom ordering on a lattice.  The definition is related to and inspired by the definition of a strongly shellable simplicial complex of Guo, Shen, and Wu in \cite{GSW19}, where they showed that the independence complex of a matroid is strongly shellable.  We prove a parallel result (Theorem \ref{geometric:lattices:are:meet-shellable}) that the lattice of flats of a matroid admits a total coatom ordering.  This result is essential to creating the Koszul filtration on the Chow ring of a matroid in Theorem \ref{Chow:rings:of:matroids:are:Koszul}.  

\begin{defn}
Let $\LL$ be a lattice, and let $\prec$ be any total order on $\LL$.  Given an element $F \in \LL$, a set $\mathcal{G} \subseteq \coat(F)$ is called an \term{initial segment covered by $F$} if for all $G, G' \in \coat(F)$ such that $G' \in \mathcal{G}$ and $G \prec G'$, then $G \in \mathcal{G}$.
\end{defn}

\begin{defn} \label{meet-shellable}
We say that a finite lattice $\LL$ admits \term{total coatom ordering} if there is a total order $\prec$ on $\LL$ with the following properties for all $F \in \LL$:
\begin{enumerate}[label = (\roman*)]
\item  \label{locally:rooted}
For all $G, G' \in \coat(F)$, if $G \prec G'$, then there exists $G'' \in \coat(F)$ such that:
\begin{enumerate}[label = (i.\alph*)]
\item  $G'' \prec G'$, 
\item $G' \wedge G'' \in \coat(G')$, and 
\item $G \wedge G' \leq G''$.
\end{enumerate}
\item \label{good:restricted:covers}
If $\mathcal{G} \subseteq \LL$ is an initial segment covered by $F$ and $G'$ is the largest element in $\mathcal{G}$ with respect to $\prec$, then the set 
\[
\coat_\mathcal{G}(G') := \{ G \wedge G' \mid G \in \mathcal{G} \text{ and } G \wedge G' \in \coat(G') \}
\]
is an initial segment covered by $G'$. 
\end{enumerate}
\end{defn}

\noindent The first defining property of a total coatom order can be visualized in the Hasse diagram of $\LL$ as shown in Figure~\ref{figure:tco}; the dashed lines represent relations in the lattice that are not necessarily covering relations.

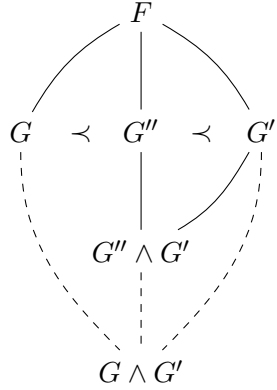
\begin{figure}[th]
\begin{tikzpicture}[scale = 0.8]
\node (F) at (0, 4) [above] {$F$};
\node (G1) at (2, 2) [above] {$G'$};
\node (G2) at (0, 2) [above] {$G''$};
\node (G3) at (-2, 2) [above] {$G$};
\node (G12) at (0, 0) [above] {$G'' \wedge G'$};
\node (G13) at (0, -2) [above] {$G \wedge G'$};
\draw (G2) -- (F);
\draw (G1) to[out = 120, in = -30] (F);
\draw (G3) to[out = 60, in = -150] (F);
\draw (G1) to[out = -120, in = 30] (G12);
\draw (G2) -- (G12);
\node at (1,2) [above] {$\prec$};
\node at (-1,2) [above] {$\prec$};
\draw[dashed] (G12) -- (G13);
\draw[dashed] (G1) to[out = -90, in = 45] (G13);
\draw[dashed] (G3) to[out = -90, in = 135] (G13);
\end{tikzpicture}
\caption{A total coatom order}\label{figure:tco}
\end{figure}

First, we show that geometric lattices have the above property.

\begin{thm} \label{geometric:lattices:are:meet-shellable}
Geometric lattices admit total coatom orderings.
\end{thm}

\begin{proof}
Let $\LL$ be a geometric lattice.  We may assume that $\LL$ is the lattice of flats of some simple matroid $M$ on the ground set $E = [n] = \{1, 2, \dots, n \}$.  Given flats $F' = \{i_1 > \cdots > i_r\}$ and $F = \{j_1 > \cdots > j_s\}$ of $M$, we will say that $F' \succ F$ if $\rk F' > \rk F$ or if $\rk F' = \rk F$ and for some $\ell \leq \min\{r, s\}$ we have $i_p = j_p$ for all $p < \ell$ and $i_\ell < j_\ell$.  We claim that $\prec$ is the desired total coatom order of $\LL$.  

First, we observe $\prec$ is compatible with restriction to flats of $M$ in the sense that  the induced order on the lattice of flats of $M\vert F$ is the analogously defined order on $M\vert F$ for any flat $F \in \LL$.  Hence, it suffices to prove that properties (i) and (ii) of Definition \ref{meet-shellable} are satisfied when $F = E$ so that $G$ and $G'$ are hyperplanes of $M$.  Since $\LL$ is a geometric lattice, we also note that saying a flat $F$ covers a flat $G$ is equivalent to saying that $G \subseteq F$ and $\rk G = \rk F - 1$.

\ref{locally:rooted} Let $H'$ and $H$ be hyperplanes of $M$ such that $H \prec H'$.  Let $X$ be a basis for $H \cap H'$ (in other words, a basis for the matroid $M\vert(H \cap H')$).  Then $\cl(X) = H \cap H'$ so that $\abs{X} = \rk(X) = \rk(H \cap H')$.  We can then find a set $Y \subseteq E$ with $X \cap Y = \emptyset$ such that $X \cup Y$ is a basis for $H'$. 

Suppose that  $H' = \{i_1 > \cdots > i_r\}$ and $H = \{j_1 > \cdots > j_s\}$.  By assumption, there is an $\ell \leq \min\{r, s\}$ such that $i_\ell < j_\ell < j_{\ell - 1} = i_{\ell -1}$ so that $j_\ell \notin H'$.  Hence, $X \cup Y \cup j_\ell$ is a basis for $M$.  We note that $\abs{Y} = \rk H' - \rk(H \cap H') \geq 1$.  Consider the hyperplane $H'' = \cl(X \cup (Y \setminus y) \cup j_\ell)$ for some $y \in Y$, and suppose that $H'' = \{k_1> \cdots > k_t\}$.  Since $X \cup (Y \setminus y) \subseteq H \cap H''$, we clearly have $\rk(H \cap H'') = \rk H - 1$ by construction, and $\{j_1, \dots, j_{\ell - 1}\} \subseteq H \cap H' = \cl(X) \subseteq H''$ so that $k_i \geq j_i$ for all $i \leq \ell$.  It follows that for some $h \leq \ell$ we have $k_h > i_h$ and $k_p = i_p$ for all $p < h$ so that $H'' \prec H'$.

\ref{good:restricted:covers} Let $\mathcal{H}$ be an initial segment of hyperplanes of $M$ with largest element $H' = \{i_1 > i_2 > \cdots > i_r\}$.  We must prove that the set $\coat_\mathcal{H}(H')$ of induced hyperplanes in $M\vert H'$ is also an initial segment.  Let $F'' \prec F$ be hyperplanes of $M\vert H'$ with $F \in \coat_\mathcal{H}(H')$.  We must show that $F'' \in \coat_\mathcal{H}(H')$.  As $F \in \coat_\mathcal{H}(H')$, we know there is a hyperplane $H \in \mathcal{H} \setminus \{H'\}$ with $F = H \cap H'$.    Since $H \prec H'$ by assumption, we know that there is an index $\ell$ and element $j > i_\ell$ of the ground set of $M$ such that 
\[ 
H = \{i_1 > \cdots > i_{\ell - 1} > j > j_{\ell + 1} > \cdots > j_s \}.
\]  
In particular, we note that $j \notin H'$ and $F = \{i_1 > \cdots > i_{\ell -1} > i_{h_\ell} > \cdots > i_{h_t} \}$ for some indices $h_k$.  Since $F'' \prec F$, we similarly know that $F''$ contains an element $e$ not in $F$.  As $F'' \setminus F \subseteq H' \setminus F$, it follows that $i_\ell \geq e$ so that $F'' = \{ i_1 > \cdots > i_{\ell -1} > i_{h_\ell} > \cdots > i_{h_{p-1}} > e > e_{p+1} > \cdots > e_q \}$.

Consider the hyperplane $H'' = \cl(F'' \cup j)$.  On the one hand, we have $F'' \subseteq H' \cap H''$ so that $\rk H' - 1 = \rk F'' \leq \rk(H' \cap H'')$.  On the other hand, we note that $H'' \neq H'$ since $j \notin H'$, and so, it follows that $\rk(H' \cap H'') \leq \rk M - 1$ since the lattice of flats of $M$ is semimodular.  Hence, $\rk F'' = \rk (H' \cap H'')$ so that $F'' = H' \cap H''$, and so, we are done if we can show that $H'' \prec H'$.  If this is not the case, then $H' \prec H''$ so that there is an index $m$ such that $H'' = \{i_1 > \cdots > i_{m-1} > k_m > \cdots > k_b \}$, where $i_m > k_m$. In particular, $i_m \notin H''$ so that $m \geq \ell$.  If $m > \ell$, then $j \in H''$ and $j > i_\ell \geq i_{m-1}$, which is impossible.  Thus, we must have $m = \ell$ and $k_m \geq j > i_\ell$, which implies that $H'' \prec H'$, a contradiction.  Hence, $H'' \prec H'$ as wanted.
\end{proof}

\begin{example}
Consider the uniform matroid $U_{5,5}$ on the ground set $E = \{1, 2, 3, 4, 5\}$.  For simplicity of notation, we will write subsets of $E$ as strings of their elements.  Since every subset of $E$ is independent, the hyperplanes of $U_{5, 5}$ are precisely the subsets of size 4.  Take $M = T_{H'}(U_{5, 5})$ to be the principal truncation of $U_{5, 5}$ with respect to the hyperplane $H' = 1234$.  (See \cite[p.279]{Oxl11} for a definition of principal truncation.) Equivalently, $M$ can be viewed as just the matroid of linearly independent columns of the matrix \vspace{0.5 em}
\[
\hspace{-2 em}
A = \begin{pmatrix} 
1 & 0 & 0 & 1 & 0 \\
0 & 1 & 0 & 1 & 0 \\
0 & 0 & 1 & 1 & 0 \\
0 & 0 & 0 & 0 & 1 
\end{pmatrix},
\vspace{0.5 em}
\]
where the columns are indexed $1$ to $5$.  The lattice of flats of $M$ is as shown in Figure~\ref{figure:trunc}. The flats of $M$ have been arranged in descending order according to the total coatom order of the preceding theorem when read from top to bottom and left to right.

\begin{figure}[th]
\lattice[baseline=(current bounding box.west), scale = 0.9]{
0/4/12345/0,
-3/3/1234/1,
-2/3/125/2,
-1/3/135/3,
0/3/235/4,
1/3/145/5,
2/3/245/6,
3/3/345/7,
-2.5/1/12/8,
-2/1/13/9,
-1.5/1/23/10,
-1/1/14/11,
-0.5/1/24/12,
0/1/34/13,
1/1/15/14,
1.5/1/25/15,
2/1/35/16,
2.5/1/45/17,
-2/0/1/18,
-1/0/2/19,
0/0/3/20,
1/0/4/21,
2/0/5/22,
0/-1/$\emptyset$/23}{0/1, 0/2, 0/3, 0/4, 0/5, 0/6, 0/7, 8/1, 8/2, 9/1, 9/3, 10/1, 10/4, 11/1, 11/5, 12/1, 12/6, 13/1, 13/7, 14/2, 14/3, 14/5, 15/2, 15/4, 15/6, 16/3, 16/4, 16/7, 17/5, 17/6, 17/7, 18/8, 18/9, 18/11, 18/14, 19/8, 19/10, 19/12, 19/15, 20/9, 20/10, 20/13, 20/16, 21/11, 21/12, 21/13, 21/17, 22/14, 22/15, 22/16, 22/17, 23/18, 23/19, 23/20, 23/21, 23/22}
\caption{The lattice $\LL(M)$}\label{figure:trunc}
\end{figure}
\end{example}

 A finite lattice $\LL$ with $\hat{0} \neq \hat{1}$ admits a \term{recursive coatom ordering} if there is a total order $\prec = \prec_{\hat{1}}$ on the set of coatoms of $\LL$ such that:
\begin{enumerate}[label = (\roman*)]

\item For all $H, H' \in \coat(\LL)$, if $H \prec H'$ and $G \in \LL$ such that $G < H, H'$, then there exists an $H'' \in \coat(\LL)$ such that:
\begin{enumerate}[label = (i.\alph*)]
\item  $H'' \prec H'$, 
\item There exists $F \in \coat(H')$ such that $G \leq F < H''$
\end{enumerate}

\item For each $H' \in \coat(\LL)$, the interval $[\hat{0},H']$ admits a recursive coatom ordering $\prec_{H'}$ such that if $F \in \coat(H) \cap \coat(H')$ for some $H \prec H'$ and $F' \in \coat(H') \setminus \bigcup_{H'' \prec H} \coat(H'')$, then $F \prec_{H'} F'$.

\end{enumerate}
By convention, we also consider the lattice with $\hat{0} = \hat{1}$ to admit a recursive coatom order.

It is known that geometric lattices admit recursive atom orderings \cite[Theorem 7.2]{WW86}, the dual notion to recursive coatom orderings.  We could not find a reference in the literature for the following result and include a proof for completeness.  The following argument was communicated to us by Vic Reiner.

\begin{prop}\label{rco}
If $\LL$ is a geometric lattice, then $\LL$ admits a recursive coatom ordering.
\end{prop}

\begin{proof}  If $\LL$ is geometric, then $\LL$ is SL-shellable by \cite[Theorem 3.7]{Bjo80}.  The dual $\LL^\ast$ of the  SL-shellable lattice  $\LL$ is also SL-shellable by \cite[Proposition 3.5]{Bjo80}.  By definition, $\LL^\ast$ is then EL-shellable \cite[Definition 3.4]{Bjo80} and hence CL-shellable by \cite[Proposition 2.3]{BW83}.  Since $\LL^\ast$ is CL-shellable, $\LL^\ast$ admits a recursive atom ordering by \cite[Theorem 5.11]{BW96a}, which is equivalent to $\LL$ admitting a recursive coatom ordering.
\end{proof}

\noindent We refer the reader to the cited references for precise definitions of the variants of shellability; we will not need them for the remainder of this paper.

It is clear from the definition that a total recursive coatom ordering induces a recursive coatom ordering.  That the reverse also holds is essentially the content of \cite[Lemma 2.10]{Del08} restricted to the case of a graded lattice.

\begin{prop}\label{Delucchi} Let $\LL$ be a graded lattice.  Then $\LL$  admits a recursive coatom orderings if and only if $\LL$ admits a total coatom ordering.
\end{prop}

\noindent We refer the reader to \cite{Del08} for a proof.  Thus one could replace Theorem~\ref{geometric:lattices:are:meet-shellable} by appealing to Propositions~ \ref{rco} and \ref{Delucchi}; however, we find our argument more constructive and direct.
By \cite[Theorem 3.3]{BW82}, CL-shellable lattices (or more generally, graded posets) are shellable.  Thus, graded lattices with a total coatom order are shellable.

\section{The Koszul property of the Chow ring of a matroid}\label{Smain}

The aim of this section is to show that the Chow ring of any matroid is Koszul by building a suitable Koszul filtration. First, we describe a presentation of $\Chow(M)$ and a corresponding Gr\"obner basis better suited to our purposes.

\subsection{The Atom-free presentation of the Chow Ring of a Matroid}\label{SSchow}

Let $M$ be a simple matroid with finite ground set $E$ and lattice of flats $\LL = \LL(M)$.  Abusing notation slightly, we identify elements $i \in E$ with the rank-one flats $\{i\}$ of $M$.  There are now several presentations of the Chow ring of a matroid.  The \term{Feichtner-Yuzvinsky presentation} of the Chow ring of $M$ \cite{FY04} is the ring
\[ \Chow_{FY}(M) = \QQ\left[x_F \mid F \in \LL \smallsetminus \{\emptyset\} \right]/I_{FY}(M), \]
where
\[ I_{FY}(M) = \left(x_Fx_{F'} \mid F, F' \; \text{incomparable}\;\right) + \left(\textstyle \sum_{i \in F}\, x_F \mid i \in E\right). \] 

In several papers \cite{AHK18, BHMPW20a, BHMPW20b}, the variable $x_E$ is eliminated so that  variables are associated to nonempty, proper flats.  Instead, we eliminate the variables associated to the atoms of $\LL$ to define a quadratic presentation.

\begin{defn}
Let $\LL_{\geq 2}$ denote the set of flats of $M$ of rank at least 2.  We define the \term{atom-free presentation} of the Chow ring of $M$ to be the ring
\[ \Chow_\af(M) = \QQ[x_F \mid F \in \LL_{\geq 2} ]/I_\af(M), \] 
where
\begin{align*} 
I_\af(M) 
&= \left(x_Fx_{F'} \mid F, F' \in \LL_{\geq 2} \; \text{incomparable}\;\right)  + \left( \textstyle x_F \sum_{F' \supseteq F \vee i} \, x_{F'} \,\bigg{\rvert}\, F \in \LL_{\geq 2}, i \in E \setminus F \right) \\
&+ \left(\textstyle \sum_{F \supseteq i \vee j} \, x_F^2 + \sum_{F' \supsetneq F \supseteq i \vee j}  \, 2x_Fx_{F'} \,\bigg{\rvert}\, i, j \in E, i \neq j\right).
\end{align*}
\end{defn}

After a change variables $\phi$ on the ring $\QQ[x_F \mid F \in \LL \setminus \{\emptyset\}]$ sending $x_i \mapsto x_i - \sum_{F \supsetneq i} x_F$ for each $i \in E$, we see that 
\begin{align} 
\begin{split} \label{change:of:variables}
\phi(I_{FY}(M)) &= \left(x_i \mid i \in E) + (x_Fx_{F'} \mid F, F' \in \LL_{\geq 2} \; \text{incomparable}\;\right)  \\ 
&+ \left( \textstyle x_F\sum_{F' \supsetneq i} \, x_{F'} \,\bigg{\rvert}\, F \in \LL_{\geq 2}, i \in E \setminus F \right) + \left(\textstyle \sum_{F \supsetneq i, F' \supsetneq j} \, x_Fx_{F'} \,\bigg{\rvert}\, i, j \in E, i \neq j \right) 
\\
&= \left(x_i \mid i \in E \right) + I_\af(M),
\end{split}
\end{align}
so that $\Chow_{FY}(M) \iso \QQ[x_F \mid \LL \setminus \{\emptyset\}]/((x_i \mid i \in E) + I_\af(M)) \iso \Chow_\af(M)$.
There is also a simplicial presentation of the Chow ring of a matroid \cite{BES20}, but we will not need it here.

\begin{rmk}
Although we use $\QQ$ as our coefficient field in this paper to keep with the terminology established in \cite{BHMPW20b}, our constructions work equally well over any coefficient field.
\end{rmk}

While a quadratic Gr\"obner basis for $I_{FY}(M)$ or $I_{\af}(M)$ in not known, a non-quadratic Gr\"obner basis for $I_{FY}(M)$ was computed by Feichtner and Yuzvinsky, which we recall here.

\begin{thm}[{\cite[Theorem 2]{FY04}}]\label{FYgb}
With respect to any lexicographic order $>$ such that $x_F > x_G$ implies $F \nsupseteq G$, the ideal $I_{FY}(M)$ has a Gr\"obner basis consisting of the following polynomials for all $F, F' \in \LL \setminus \{\emptyset\}$:
\begin{align} 
&x_{F'}x_F& F, F' \text{incomparable} \\
&x_{F'}\left(\textstyle \sum_{G \supseteq F}\, x_G\right)^{\rk F - \rk F'}& F' \subsetneq F \\
&\left(\textstyle \sum_{G \supseteq F}\, x_G\right)^{\rk F} 
\end{align}
\end{thm}

A slight modification yields a non-quadratic Gr\"obner basis in the atom-free setting.

\begin{cor} \label{atom-free:Gröbner:basis}
With respect to any lexicographic order $>$ such that $x_F > x_G$ implies $F \nsupseteq G$, the ideal $I_{\mathrm{af}}(M)$ has a Gr\"obner basis consisting of the following polynomials for all $F, F' \in \LL_{\geq 2}$:
\begin{align} 
&x_Fx_{F'}& F, F' \text{incomparable} \\
&x_{F'}\left(\textstyle \sum_{G \supseteq F}\,  x_G\right)^{\rk F - \rk F'}& F' \subsetneq F \\
&\left(\textstyle \sum_{G \supseteq F}\,  x_G\right)^{\rk F} & 
\end{align}
\end{cor}

\begin{proof}
As in the proof of \cite[3.2.2]{BES20}, we note that the automorphism $\phi$ of $\QQ[x_F\mid F \in \LL \setminus \{\emptyset\}]$ of \eqref{change:of:variables} satisfies $\init_> \phi(x_F) = x_F$ for all flats $F$; hence, $\init_> \phi(f) = \init_> f$ for all polynomials $f$ so that $I_{FY}(M)$ and $I_{\mathrm{af}}(M)$ have the same initial ideal and the images under $\phi$ of the Gr\"obner basis for $I_{FY}(M)$ form a Gr\"obner basis for $\phi(I_{FY}(M)) = (x_i \mid i \in E) + I_{\mathrm{af}}(M)$.  This Gr\"obner basis includes the polynomials listed above and the polynomials:
\begin{align} \label{superfluous:Gröbner:basis:1}
&x_F\left(\textstyle x_i - \sum_{F' \supsetneq i} \, x_{F'}\right)& F \in \LL_{\geq 2}, i \in E, F \nsupseteq i \\
\label{superfluous:Gröbner:basis:2}
&\left(\textstyle x_i - \sum_{F' \supsetneq a} \, x_{F'}\right)\left(\textstyle \sum_{G \supseteq F} x_G\right)^{\rk F - 1}& i \in E, i \subsetneq F \\
& x_i & i \in E
\end{align}
From this, we see that 
\begin{align} \label{Chow:ring:initial:ideal} 
\begin{split}
\init_>(\phi(I_{FY}(M)) &= \left(x_i \mid i \in E \right) + \left(x_Fx_{F'} \mid F, F' \in \LL_{\geq 2} \; \text{incomparable}\right) \\
&+ \left(x_{F'}x_F^{\rk F - \rk F'} \mid F, F' \in \LL_{\geq 2}, F' \subsetneq F\right) + \left(x_F^{\rk F} \mid F \in \LL_{\geq 2}\right),
\end{split}
\end{align}
which shows that omitting the above polynomials in \eqref{superfluous:Gröbner:basis:1}--\eqref{superfluous:Gröbner:basis:2} still leaves a Gr\"obner basis for $\phi(I_{FY}(M))$.

Since $<$ is an elimination order for the variables $x_i$ with $i \in E$, it follows that the polynomials listed in the statement of the corollary are a Gr\"obner basis for the elimination ideal $\phi(I_{FY}(M)) \cap \QQ[x_F \mid F \in \LL_{\geq 2}]$ by \cite[3.3]{EH12}.  It then suffices to note that $\phi(I_{FY}(M)) \cap \QQ[x_F \mid F \in \LL_{\geq 2}] = \pi((x_i \mid i \in E) + I_\af(M)) = I_\af(M)$, where $\pi: \QQ[x_F \mid F \in \LL \setminus \{\emptyset\}] \to \QQ[x_F \mid F \in \LL_{\geq 2}]$ is the natural algebra retract sending $x_i \mapsto 0$ for all $i \in E$.
\end{proof}

Combining \eqref{Chow:ring:initial:ideal} with Macaulay's Theorem yields the following.

\begin{cor} \label{monomial:basis}
The ring $\Chow_{\mathrm{af}}(M)$ has a $\QQ$-basis consisting of all monomials
\[ x_F^\alpha = x_{F_1}^{\alpha_1} \cdots x_{F_r}^{\alpha_r} \]
for all chains of flats $F = \{F_1 \supsetneq F_2 \supsetneq \cdots \supsetneq F_r \supsetneq F_{r+1} = \emptyset \}$ for some $r \geq 0$ with $F_i \in \LL_{\geq 2}$ for all $i \leq r$ and all $\alpha \in \ZZ_{\geq 0}^{r+1}$ such that $\sum_i \alpha_i = \rk F_1$ and $1 \leq \alpha_i < \rk F_i - \rk F_{i + 1}$ if $i \leq r$.
\end{cor}

\noindent Following \cite{FY04} and \cite{BES20}, we call the monomials of Corollary \ref{monomial:basis} the \term{nested monomials} of $\Chow_\af(M)$.  In this notation, we note that $x_{\{\emptyset\}}^{\mathbf{0}} = 1$.

Adiprasito, Huh, and Katz \cite{AHK18} proved that $\Chow(M)$ has the \term{K\"ahler package}; that is, $\Chow(M)$ is an Artinian, Gorenstein $\mathbb{Q}$-algebra of regularity $\reg(\Chow(M)) = \rk(M) - 1$ with versions of the hard Lefschetz and Hodge-Riemann conditions.  Other proofs of these properties have now been given in \cite{BES20} and \cite{BHMPW20a}.   It is well known that an Artinian, graded $\kk$-algebra $A$ is a Poincar\'e duality algebra if and only if $A$ is Gorenstein; see e.g. \cite[Proposition 2.1]{MW09}.   

\begin{rmk} \label{socle:generator}
As a consequence of Corollary \ref{atom-free:Gröbner:basis}, we note that $I_{\mathrm{af}}(M)$ contains the polynomials
\begin{equation} \label{hyperplane:relations}
x_Hx_E, \qquad x_Fx_H^{\rk H - \rk F} + x_Fx_E^{\rk M - \rk F - 1}, \qquad x_H^{\rk H} + x_E^{\rk M -1}
\end{equation}
for any $F \in \LL_{\geq 2}$ and hyperplane $H \supsetneq F$.  Moreover, we have $x_Fx_E^{\rk M - 1} = 0$ in $\Chow_\af(M)$ for all $F \in \LL_{\geq 2}$.  Since $\Chow_\af(M)$ is an Artinian Gorenstein ring, it follows that $x_E^{\rk M - 1}$ is a socle generator of $\Chow_\af(M)$.  These observations will be useful in the computations below.
\end{rmk}

\subsection{A Koszul filtration for the Chow ring of a matroid}\label{SSfilt1}

 We are now in position to describe the Koszul filtration.
 When $M$ is a matroid of rank 1 or 2, we note that $\Chow(M) \iso \QQ$ or $\Chow(M) \iso \QQ[x_E]/(x_E^2)$, respectively.  In either case, the Chow ring is clearly Koszul, and moreover, it is has a Koszul filtration.  In the former case, the filtration consists of only the zero ideal, and in the latter case, the filtration consists of the zero ideal and the graded, maximal ideal.   Hence, in what follows, we may assume the following without any loss of generality.

\begin{notation}
Throughout the remainder of this section, $M$ denotes a simple matroid with $\rk M \geq 3$ and $A = \Chow(M)$ with respect to the atom-free presentation.
\end{notation}

Although we will not exactly present the proof this way, the idea behind the construction of the desired Koszul filtration is that we can use natural matroid operations such as truncation and restriction to reduce the rank of our matroid down to the low-rank cases mentioned above and inductively lift a Koszul filtration to the Chow ring of our arbitrary matroid $M$. Recall that a flat $H$ of $M$ is called a hyperplane if $\rk H = \rk M - 1$, and the hyperplanes are precisely the coatoms in the lattice of flats $\LL = \LL(M)$.  Since the hyperplanes are exactly the flats omitted when truncating the matroid $M$, we are naturally led to study ideals generated by hyperplane variables.

\begin{lemma} \label{hyperplane:ideal:basis}
Let $\mathcal{H}$ be a set of hyperplanes of $M$.  Then the ideal $(x_H \mid H \in \mathcal{H})$ has a $\QQ$-basis consisting of all nested monomials $x_F^\alpha$ with $F = \{F_1 \supsetneq F_2 \supsetneq \cdots \supsetneq F_r \supsetneq F_{r+1} = \emptyset \}$ for some $r \geq 0$ such that either: 
\begin{enumerate}[label = \textnormal{(\roman*)}]
\item $F_1 \in \mathcal{H}$, or 
\item $F_1 = E$, $H \supsetneq F_2$ for some $H \in \mathcal{H}$, and $\alpha_1 = \rk M - \rk F_2 - 1$. 
\end{enumerate}
\end{lemma}

\begin{proof}
Consider the case of a single hyperplane $H$.  The general case follows easily from the principal case.  If $F_1 = H$, then clearly $x_F^\alpha \in (x_H)$.  If we have $F_1 = E$, $H \supsetneq F_2$, and $\alpha_1 = \rk M - \rk F_2 - 1$, then by \eqref{hyperplane:relations} we have $x_E^{\alpha_1}x_{F_2} = - x_H^{\alpha_1}x_{F_2}$ provided $F_2 \neq \emptyset$ or $x_E^{\alpha_1} = -x_H^{\alpha_1}$ if $F_2 = \emptyset$; either way, it follows that $x_F^\alpha \in (x_H)$.  Hence, $(x_H)$ contains all monomials in the statement of the lemma.  Conversely, if $f \in (x_H)$, then $f = x_Hg$ for some $g \in \Chow(M)$.  Write $g = \sum c_{F, \alpha} x_F^\alpha$ as a $\QQ$-linear combination of nested monomials.  For any monomial $x_F^\alpha$ in the support of $g$, we see that $x_Hx_F^\alpha$ is nonzero only if $F_1 \subseteq H$.  Indeed, if $F_1 \nsubseteq H$, then either $F_1$ and $H$ are incomparable or $F_1 = E$.  Either way, the defining relations of $\Chow(M)$ or \eqref{hyperplane:relations} imply that $x_{F_1}x_H = 0$.  If $F_1 = H$, then $x_Hx_F^\alpha$ is a scalar multiple of a nested monomial, even if $\alpha_1 = \rk H - \rk F_2 - 1 = \rk M - \rk F_2 - 2$ since then $x_Hx_F^\alpha = -x_E^{\rk M - \rk F_2 -1}x_{F_2}^{\alpha_2} \cdots x_{F_r}^{\alpha_r}$.  If $H \supsetneq F_1$ and $\rk H - \rk F_1 \geq 2$, then $x_Hx_F^\alpha$ is also a nested monomial.  If $H \supsetneq F_1$ and $\rk H - \rk F_1  = 1$, then $x_Hx_F^\alpha = -x_Ex_{F_1}^{\alpha_1} \cdots x_{F_r}^{\alpha_r}$ is also a nested monomial.  Thus, we can write 
\[ f = \sum_{G_1 = H} c_{G, \beta} x_G^\beta + \sum_{H \supsetneq G_2} c_{G, \beta} x_E^{\rk M - \rk G_2 - 1}x_{G_2}^{\beta_2} \cdots x_{G_r}^{\beta_r}. \qedhere
\]
\end{proof}

Since $A = \Chow(M)$ is an Artinian, Gorenstein ring, it follows by linkage \cite[Remark 2.7]{HU87} that the quotient ring $A/(0 : m)$ is also Gorenstein for any monomial $m$, and one might hope that this ring is the Chow ring of a matroid quotient of $M$.  We identify the matroids associated to such quotients in some simple but important cases.

\begin{prop} \label{truncation} 
For a matroid $M$, we have $(0 : x_E) = (x_H \mid H \in \coat(E))$ and $A/(0 : x_E) \iso \Chow(T(M))$. 
\end{prop}

\begin{proof}
Clearly, the latter ideal is contained in the former by \eqref{hyperplane:relations}.  Suppose that $f \in A$ with $x_Ef = 0$, and write $f = \sum c_{F, \alpha} x_F^\alpha$ as a $\QQ$-linear combination of nested monomials.  For each nested monomial $x_F^\alpha$ in $A$, we note that $x_Ex_F^\alpha = 0$ only if $F_1$ is a hyperplane or $F_1 = E$ and $\alpha_1 = \rk M - \rk F_2 - 1$ (including the possibility that $F_2 = \emptyset$).  Otherwise, $x_Ex_F^\alpha$ is still a nested monomial.  It then follows from the equality $x_Ef = 0$ that 
\[ 
f = \sum_{F_1 \in \coat(E)} c_{F, \alpha} x_F^\alpha + \sum_{F_1 = E, \alpha_1 = \rk M - \rk F_2 - 1} c_{F, \alpha}x^\alpha_F.  
\]
For each monomial in the latter sum, if $F_2 \neq \emptyset$, we note that $\rk F_2 \leq \rk M - 2$ so that there is a hyperplane $H \supsetneq F_2$ as the lattice of flats of $M$ is graded, and $x_E^{\rk M - \rk F_2 - 1}x_{F_2} = -x_H^{\rk H - \rk F_2}x_{F_2}$.  Otherwise, if $F_2 = \emptyset$, we have $x_E^{\rk M - 1} = -x_H^{\rk H}$ for any hyperplane $H$.  This shows that $f \in (x_H \mid H \in \coat(E))$.

It is easily checked that there is a well-defined surjective $\QQ$-algebra map $\pi: A \to \Chow(T(M))$ defined by $x_F \mapsto x_F$ for $F \notin \coat(E)$ and $x_F \mapsto 0$ otherwise.  We will show that $\ker \pi = (x_H \mid H \in \coat(E))$.  Again, it is clear that the latter ideal is contained in the former.  If $f \in \ker \pi$ and we write $f = \sum_\alpha c_{F, \alpha}x_F^\alpha$ as a $\QQ$-linear combination of nested monomials, then 
\[
0 = \pi(f) = \sum_{F_1 \notin \coat(E)} c_{F, \alpha} x_F^\alpha = \sum_{\rk F_1 \leq \rk M -2} c_{F, \alpha} x_F^\alpha + \sum_{F_1 = E, \alpha_1 < \rk M - \rk F_2 - 1} c_{F, \alpha} x_F^\alpha 
\]
expresses $\pi(f)$ as a linear combination of nested monomials in $\Chow(T(M))$.  In particular, note that the latter sum does not include any nested monomials with $\rk F_2 = \rk M - 2$ since $F_2$ is a hyperplane in $T(M)$ so that $x_{F_2}x_E = 0$ by \eqref{hyperplane:relations}.  Hence, the coefficients of all the above monomials in $f$ must be zero so that 
\[
f = \sum_{F_1 \in \coat(E)} c_{F, \alpha} x_F^\alpha +  \sum_{F_1 = E, \alpha_1 =  \rk M - \rk F_2 -1} c_{F, \alpha} x_F^\alpha. 
\]
However, for every nonempty flat $F_2$ we note that $x_{F_2}x_E^{\rk M - \rk F_2 - 1} = -x_{F_2}x_H^{\rk H - \rk F_2}$ for any hyperplane $H \supsetneq F_2$, and $x_E^{\rk M -1} = -x_H^{\rk H}$ for any hyperplane $H$.  Thus, we see that $f \in (x_H \mid H \in \coat(E))$ as wanted, and combining this with the preceding paragraph yields $A/(0 : x_E) \iso \Chow(T(M))$.
\end{proof}

\begin{cor} \label{colons:of:arbitrary:flats}
Let $F$ be a flat of $M$ with $\rk F \geq 2$.
\begin{enumerate}[label = \textnormal{(\alph*)}]
\item \label{kernel:of:restriction} The kernel of the natural surjective homomorphism $\pi_{M,F}: A \to \Chow(M \vert F)$ sending $x_G \mapsto x_G$ if $G \subseteq F$ and $x_G \mapsto 0$ otherwise is $(x_G \mid G \nsubseteq F)$.
\item \label{colon:with:filter} Suppose that $\mathcal{G}$ is a set of flats of $M$ such that $(F, E] \subseteq \mathcal{G}$ and $\mathcal{G} \cap [\emptyset, F] = \emptyset$. If $\rk F \geq 3$, then
\[ 
(x_G \mid G \in \mathcal{G}) : x_F = (x_G \mid G \nsubseteq F) + (x_G \mid G \in \coat(F)). 
\]
Otherwise, $(x_G \mid G \in \mathcal{G}) : x_F = A_+$.
\end{enumerate}
\end{cor}

\begin{proof}
\ref{kernel:of:restriction}  It is easily checked that there is a well-defined surjective $\QQ$-algebra map $\pi: A \to \Chow(M\vert F)$ defined by the given assignments.  We will show that $\ker \pi = (x_G \mid G \nsubseteq F)$.  Again, it is clear that the latter ideal is contained in the former.  If $f \in \ker \pi$ and we write $f = \sum c_{G, \alpha} x_G^\alpha$ as a $\QQ$-linear combination of nested monomials, then $0 = \pi(f) = \sum_{G_1 \subseteq F} c_{G, \alpha} x_G^\alpha$ expresses $\pi(f)$ as a linear combination of nested monomials in $\Chow(M\vert F)$.  Hence, the coefficients of all the preceding monomials in $f$ must be zero so that $f = \sum_{G_1 \nsubseteq F} c_{G, \alpha} x_G^\alpha \in (x_G \mid G \nsubseteq F)$ as wanted. 

\ref{colon:with:filter}  First, we note that the colon ideal contains $(x_G \mid G \nsubseteq F)$.  Indeed, if $G'$ is flat such that $G' \nsubseteq F$, then either $F \subsetneq G'$ so that $x_{G'} \in (x_G \mid G \in \mathcal{G})$, or $G'$ is incomparable with $F$ so that $x_{G'}x_F = 0$.  Since both ideals in question contain $(x_G \mid G \nsubseteq F)$, it suffices by part \ref{kernel:of:restriction} to show that their images agree in $\Chow(M\vert F)$.  Hence, it suffices to assume $F = E$ and to prove that $(0 : x_E) = (x_H \mid H \in \coat(E))$ when $\rk M \geq 3$ and $(0 : x_E) = A_+$ when $\rk M = 2$, which follow from the preceding proposition and the discussion at the beginning of this section respectively.
\end{proof}

\begin{prop} \label{colon:computations}
Let $\mathcal{H}$ be a nonempty set of hyperplanes of $M$ and $H'$ be any hyperplane.  Then:
\begin{enumerate}[label = \textnormal{(\alph*)}]

\item  \label{restriction}
$(0 : x_{H'}) = (x_F \mid F \nsubseteq H' )$ and $A/(0: x_{H'}) \iso \Chow(M\vert H')$.

\item \label{annihilator:of:hyperplane:ideal}
$(0 : (x_H \mid H \in \mathcal{H})) = (x_F \mid F \nsubseteq H \;\text{for all}\; H \in \mathcal{H})$.

\item \label{hyperplane:colons}
Suppose that $H' \notin \mathcal{H}$.  If $\rk H' \geq 3$ and for every $H \in \mathcal{H}$ there exists an $F \in \coat_\mathcal{H}(H')$ such that $H \cap H' \subseteq F$, then
\[
(x_H \mid H \in \mathcal{H}) : x_{H'} =  (0 : x_{H'}) + (x_F \mid F \in \coat_\mathcal{H}(H'))
\]
Otherwise, if $\rk H' = 2$, then $(x_H \mid H \in \mathcal{H}) : x_{H'} = A_+$. 
\end{enumerate}
\end{prop}

\begin{proof}

\ref{restriction}  We have already seen in the proof of Lemma \ref{hyperplane:ideal:basis} that for any nested monomial $x_F^\alpha$ in $A$ we have $x_{H'}x_F^\alpha = 0$ if $F_1 \nsubseteq H'$ and $x_{H'}x_F^\alpha$ is a scalar multiple of a nested monomial otherwise.  Furthermore, multiplication by $x_{H'}$ takes distinct nested monomials with $F_1 \subseteq H'$ to scalar multiples of distinct nested monomials.  In particular, it is immediate that the latter ideal is contained in the former.  Moreover, if $f \in A$ with $x_{H'}f = 0$ and we write $f = \sum c_{F, \alpha} x_F^\alpha$ as a $\QQ$-linear combination of nested monomials, it follows from the equality $x_{H'}f = 0$ that $f = \sum_{F_1 \nsubseteq H'} c_{F, \alpha} x_F^\alpha$, which shows that $f \in (x_F \mid F \nsubseteq H')$.  The isomorphism $A/(0 : x_{H'}) \iso \Chow(M\vert H')$ then follows from part \ref{kernel:of:restriction} of the preceding corollary.

\ref{annihilator:of:hyperplane:ideal}  Consider the homomorphism $\eta: A \to \bigoplus_{H \in \mathcal{H}} \Chow(M\vert H)$ sending each $f \in A$ to the tuple of its images under the natural maps $A \to \Chow(M\vert H)$.  By part \ref{restriction}, we know that $\ker \eta = \bigcap_{H \in \mathcal{H}} (0 : x_H) = (0 : (x_H \mid H \in \mathcal{H}))$.  Clearly, this ideal contains $x_F$ for every flat $F$ such that $F \nsubseteq H$ for all $H \in \mathcal{H}$.  Conversely, if $f \in \ker \eta$, then the proof of the previous part shows that $f$ has an expansion in terms of nested monomials of the form $f = \sum_{F_1 \nsubseteq H, \;\text{all}\; H \in \mathcal{H}} c_{F, \alpha} x_F^\alpha$.  Hence, $f \in (x_F \mid F \nsubseteq H \;\text{for all}\; H \in \mathcal{H})$ as wanted.

\ref{hyperplane:colons} The statement is clear if $\rk H' = 2$, since $x_{H'}x_F = 0$ for every $F \neq H'$ and $x_{H'}^2 = -x_E^2 = x_H^2 \in (x_H \mid H \in \mathcal{H})$ for any $H \in \mathcal{H}$ by \eqref{hyperplane:relations}.  So, we may assume that $\rk H' \geq 3$.  In that case, the latter ideal is again clearly contained in the former since $x_Fx_{H'} = -x_Fx_E = x_Fx_H$ if $F = H \cap H'$ and $\rk F = \rk M - 2$.  Suppose that $f \in A$ with $x_{H'}f \in (x_H \mid H \in \mathcal{H})$, and write $f = \sum c_{F, \alpha} x_F^\alpha$ as a $\QQ$-linear combination of nested monomials.  Then: 
\begin{align*}
x_{H'}f = \sum_{F_1 \subseteq H'} c_{F, \alpha} x_{H'}x_F^\alpha &= \sum_{\rk H' - \rk F_1 \geq 2} c_{F, \alpha} x_{H'}x_F^\alpha + \sum_{\substack{F_1 = H' \\ \alpha_1 < \rk H' - \rk F_2 -1}} c_{F, \alpha} x_{H'}^{\alpha_1 + 1}x_{F_2}^{\alpha_2} \cdots x_{F_r}^{\alpha_r} \\
&- \sum_{\rk H' - \rk F_1 = 1} c_{F, \alpha} x_Ex_F^\alpha - \sum_{\substack{F_1 = H' \\ \alpha_1 = \rk H' - \rk F_2 -1}} c_{F,\alpha} x_{E}^{\rk M - \rk F_2-1}x_{F_2}^{\alpha_2} \cdots x_{F_r}^{\alpha_r}
\end{align*}
By Lemma \ref{hyperplane:ideal:basis}, the coefficients of the monomials in the first two sums must all be zero, and the coefficient of a monomial in the last two sums is also zero unless $F_1 \subsetneq H$ for some $H \in \mathcal{H}$ or $F_1 = H'$ and $F_2 \subsetneq H$ for some $H \in \mathcal{H}$.  Hence, we have
\begin{align*}
f &= \sum_{F_1 \nsubseteq H'} c_{F, \alpha} x_F^\alpha + \sum_{\substack{\rk H' - \rk F_1 = 1 \\ F_1 \subsetneq H \in \mathcal{H}}} c_{F, \alpha} x_F^\alpha + \sum_{F_2 \subsetneq H \in \mathcal{H}} c_{F, \alpha} x_{H'}^{\rk H' - \rk F_2-1}x_{F_2}^{\alpha_2} \cdots x_{F_r}^{\alpha_r}.
\end{align*}
It follows that $f \in (0 : x_{H'}) + (x_Fx_{H'}^{\rk H' - \rk F -1}, x_{H'}^{\rk H' -1} \mid F \subseteq H \cap H' \;\text{for some}\; H \in \mathcal{H})$.  It remains to show that 
\[
(0 : x_{H'}) + (x_Fx_{H'}^{\rk H' - \rk F -1}, x_{H'}^{\rk H' -1} \mid F \subseteq H \cap H' \;\text{for some}\; H \in \mathcal{H})
\]
is contained in the ideal $(0 : x_{H'}) + (x_F \mid F \in \coat_\mathcal{H}(H'))$.  After quotienting by $(0 : x_{H'})$, it suffices by part \ref{restriction} to show that 
\[
(x_Fx_{H'}^{\rk H' - \rk F -1}, x_{H'}^{\rk H' -1} \mid F \subseteq H \cap H' \;\text{for some}\; H \in \mathcal{H})
\]
is contained in the ideal $J = (x_F \mid F \in \coat_\mathcal{H}(H'))$ in $\Chow(M\vert H')$.

Since $\mathcal{H} \neq \emptyset$, we have $\coat_\mathcal{H}(H') \neq \emptyset$ by assumption so that $J \neq 0$.  By Remark \ref{socle:generator}, we know $x_{H'}^{\rk H' - 1}$ generates the socle of $\Chow(M\vert H')$ and is, therefore, contained in $J$.  We must also show that $x_Fx_{H'}^{\rk H' - \rk F - 1} \in J$ if $F \subseteq H \cap H'$ for some $H \in \mathcal{H}$ and $\rk H' \geq \rk F + 2$.  For this, since the elements $\coat_\mathcal{H}(H')$ are hyperplanes of the matroid $M\vert H'$, we note that 
\[ 
(0 : J) = (x_G \mid G \subseteq H', G \nsubseteq F'\; \text{for all}\; F' \in \coat_\mathcal{H}(H'))
\] 
by part \ref{annihilator:of:hyperplane:ideal}. Because $\Chow(M\vert H')$ is Gorenstein, it follows by linkage \cite[Theorem 21.23]{Eis95} that $(0 : (0 : J)) = J$, and so, it suffices to note that each such $x_G$ annihilates $x_Fx_{H'}^{\rk H' - \rk F - 1}$ to prove that $x_Fx_{H'}^{\rk H' - \rk F - 1} \in J$.  Given a flat $G$ with $G \subseteq H'$ and $G \nsubseteq F'$ for all $F' \in \coat_\mathcal{H}(H')$, we note that $G \nsubseteq F$ since otherwise we would have $G \subseteq F \subseteq H \cap H' \subseteq F'$ for some $F' \in \coat_\mathcal{H}(H')$ by assumption.  Consequently, either $G$ and $F$ are incomparable so that $x_Fx_G = 0$ or $F \subsetneq G$ so that $x_Gx_{H'}^{\rk H' - \rk F - 1} = 0$ by Corollary \ref{atom-free:Gröbner:basis} as $\rk H' - \rk F - 1 \geq \rk H' - \rk G$.  And so, either way we have $x_Fx_Gx_{H'}^{\rk H' - \rk F - 1} = 0$ as wanted.
\end{proof}


The condition in part \ref{hyperplane:colons} of the above proposition that for every $H \in \mathcal{H}$ there exists an $F \in \coat_\mathcal{H}(H')$ such that $H \cap H' \subseteq F$ is essential for the colon ideal to be generated by linear forms.  We give some examples below showing how the colon ideal can fail to be generated by linear forms if this condition is not satisfied.

\begin{example} \label{good:sets:of:hyperplanes:obstruction:1}
Consider the uniform matroid $U_{5,6}$ on the ground set $E = \{1, 2, 3, 4, 5, 6\}$.  Since every subset of $E$ of size at most 5 is independent, the hyperplanes of $U_{5, 6}$ are precisely the subsets of size 4.  For $H' = 1256$, the proof of the proposition shows that in $\Chow(U_{5,6})$ we have
\[ 
(x_{1234}) : x_{1256} = (0 : x_{1256} ) + (x_{12}x_{1256}, x_{1256}^3)  = (0 : x_{1256} ) + (x_{12}x_{1256}) 
\]
where $x_{12}x_{1256} \notin (0 : x_{1256})$ since it is a nested monomial of $\Chow(U_{5,6}\vert H')$.

If we take $M = T_{H'}(U_{5,6})$ to be the principal truncation $U_{5,6}$ with respect to $H'$, then $M$ is a rank 4 matroid whose hyperplanes are $H'$ and all subsets of $E$ of size 3 that are not contained in $H'$.  Even in this case, we still have
\[ 
(x_{234}) : x_{1256} = (0 : x_{1256} ) + (x_{1256}^2),
\]
where $x_{1256}^2 \notin (0 : x_{1256})$ since it is a nested monomial of $\Chow(M\vert H')$.

In the first example above, the issue is that the two hyperplanes in question have rank 4 but their meet $1234 \wedge 1256 = 12$ has rank $2$.  One might hope that by choosing only sets of hyperplanes that are locally connected (in the sense that between any two hyperplanes there is a sequence of hyperplanes in the set such each pair of consecutive hyperplanes has a meet of rank $\rk M - 2$), we might avoid quadratic terms in these colon ideals.  The following example shows even this is not enough.

Again, consider $\Chow(U_{5,6})$.  Although the set of hyperplanes $\{1234,2345,3456,1456,1256,1236\}$ in $U_{5,6}$ is locally connected (the corresponding graph is a $6$-cycle), one computes that
\[
(x_{1234},x_{2345},x_{3456},x_{1456},x_{1236}):x_{1256} = (0:x_{1256}) + (x_{125},x_{126},x_{12}x_{1256}).\]
where the quadratic monomial $x_{12}x_{1256}$ is a minimal generator of this colon ideal.  \end{example}

\begin{thm} \label{Chow:rings:of:matroids:are:Koszul}
The Chow ring of a matroid has a Koszul filtration.  
\end{thm}

\begin{proof}
Let $M$ be a matroid with lattice of flats $\LL$.  As previously noted at the beginning of this subsection, we may assume without loss of generality that $M$ is simple and $\rk M \geq 3$.  By Theorem \ref{geometric:lattices:are:meet-shellable}, we also know that the lattice $\LL$ admits a total coatom order $\prec$.  In what follows, all ideals are considered with respect to the atom-free presentation of $A = \Chow(M)$, and all initial segments of flats are with respect to the chosen total coatom ordering.  However, all up-sets refer to the natural partial order of $\LL$.  Consider the collections of ideals:
\begin{align*} 
\mathcal{F}_0 &= \{(x_G \mid G \in \mathcal{G}) \mid \mathcal{G} \subseteq \LL_{\geq 2} \;\text{an up-set} \} \\
\mathcal{F}_1 &= \{(x_G \mid G \nsubseteq F) + (x_G \mid G \in \mathcal{G}) \mid F \in \LL_{\geq 2}, \;\mathcal{G} \; \text{an initial segment covered by}\; F \},
\end{align*}
and set $\mathcal{F} = \mathcal{F}_0 \cup \mathcal{F}_1$. Note that $(0) \in \mathcal{F}_0$ when $\mathcal{G} = \emptyset$ and that $A_+ \in \mathcal{F}_0$ when $\mathcal{G} = \LL_{\geq 2}$.  We claim that $\mathcal{F}$ is the desired Koszul filtration. 

Let $0 \neq I \in \mathcal{F}_0$ so that $I = (x_G \mid G \in \mathcal{G})$ for some nonempty up-set $\mathcal{G}$.  We then choose any flat $G' \in \mathcal{G}$ minimal in $\mathcal{G}$ with respect to the partial order of $\LL$ and set $\mathcal{G}' = \mathcal{G} \setminus \{G'\}$.  It is easy to see that $\mathcal{G}'$ is also an up-set since we are removing a minimal element of $\mathcal{G}$.   Consider the ideal $J = (x_G \mid G \in \mathcal{G}') \in \mathcal{F}_0$.  Clearly $J \subseteq I$ and $I/J$ is cyclic.  Since $\mathcal{G}'$ contains all flats properly containing $G'$ and none of the flats contained in $G'$, it follows from Corollary \ref{colons:of:arbitrary:flats} that
\[ (J : I) = (J : x_{G'}) = (x_G \mid G \nsubseteq G') + (x_G \mid G \in \coat(G') ) \]
if $\rk G' \geq 3$ and $(J : I) = A_+$ if $\rk G' = 2$.  In the former case, it follows that $(J : I) \in \mathcal{F}_1$, and we have $(J : I) \in \mathcal{F}_0$ in the latter case. 

Let $0 \neq I \in \mathcal{F}_1$ so that $I = (x_G \mid G \nsubseteq F) + (x_G \mid G \in \mathcal{G})$ for some flat $F$ and $\mathcal{G}$ an initial segment of flats covered by $F$.  If $\mathcal{G} \neq \emptyset$, then $\rk G \geq 3$, and we can take $G'$ to be the largest element of $\mathcal{G}$ with respect to the total coatom order and $\mathcal{G}' = \mathcal{G} \setminus \{G'\}$.  Then $J = (x_G \mid G \nsubseteq F) + (x_G \mid G \in \mathcal{G}') \in \mathcal{F}_1$.  Again, $J \subseteq I$ with $I/J$ cyclic.  Moreover
\[ (J : I) = (J : x_{G'}) = (x_G \mid G \nsubseteq G') + (x_G \mid G \in \coat_\mathcal{G}(G')) \]
by Corollary \ref{colons:of:arbitrary:flats} and part \ref{hyperplane:colons} of Proposition \ref{colon:computations} applied to $M\vert F$.  By the definition of a total coatom ordering, $\coat_\mathcal{G}(G')$ is an initial segment of flats covered by $G'$ so that $(J : I) \in \mathcal{F}_1$.  On the other hand, if $\mathcal{G} = \emptyset$, then $I = (x_G \mid G \nsubseteq F) \in \mathcal{F}_0$, and so, we have already seen in the preceding paragraph how to choose an ideal contained in $I$ with the desired properties.  Thus, $\mathcal{F}$ is a Koszul filtration for $A$.
\end{proof}

We get the following corollary to our main theorem.

\begin{cor}
The Chow ring of a matroid has a rational Poincar\'e series.
\end{cor}

\begin{proof}
It is well-known that any commutative Koszul $\kk$-algebra $A$ has a rational Poincar\'e series; indeed, the Hilbert Series $\HS_A(t)$ is always rational so that the Poincare series $\Poin_\kk^A(t) = \frac{1}{\HS_A(-t)}$ is as well by \eqref{PHS1}.
\end{proof}

\section{The Koszul property of the augmented Chow ring of a matroid}\label{Saugmented}

\subsection{The augmented Chow ring of a Matroid}\label{SSaugchow}

The \term{augmented Chow ring} of a simple matroid $M$ with lattice of flats $\LL = \LL(M)$  is the ring
\[\aChow(M) := S_M/(I_M + J_M),\]
where
\[S_M = \QQ[y_i, x_F \mid i \in E, F \in \LL \setminus \{E\} ],\]
and $I_M$ and $J_M$ are the ideals
\begin{align*}
I_M &= (y_i - \textstyle \sum_{i \notin F} \, x_F \mid i \in E ),\\[1 ex]
J_M &= (x_Fx_G \mid F,G \;\text{incomparable}\;) + (y_ix_F \mid i \in E, i \notin F).
\end{align*}
Augmented Chow rings of matroids were introduced in \cite{BHMPW20a} and employed in \cite{BHMPW20b} as a key ingredient in the resolution of the Top-Heavy Conjecture that interpolates between the combinatorics of the lattice of flats of a matroid (as encoded by its graded M\"obius algebra) and the good algebraic properties of the Chow ring of the matroid.  In particular, it is shown \cite[Theorem 2.19]{BHMPW20b} that the augmented Chow ring of a matroid is also a quadratic Artinian Gorenstein graded algebra.  It was pointed out to us by Chris Eur that the augmented Chow ring of $M$ could be viewed as the Feichtner-Yuzvinsky Chow ring $D(\LL, \mathcal{G})$ of the lattice of flats of a related matroid, the free coextension of $M$, with respect to a certain building set.  Using this observation, we show how to obtain an atom-free presentation for the augmented Chow ring of a matroid and that all of the results of the preceding section carry over with only minor modifications to show that the augmented Chow ring of a matroid is also Koszul.

We begin by sketching the details of Eur's observation about the augmented Chow ring.  Let $M$ be a simple matroid on the ground set $E$.  The \term{free coextension} of $M$ is the matroid $(M^* + e)^*$; it is the dual of the free single element extension of the dual matroid of $M$.  We refer the interested reader to \cite[Sections 2.1, 7.2]{Oxl11} for further details on duality and free extensions of matroids.  For our purposes, it suffices to note that the free coextension $(M^* + e)^*$ is a matroid on the ground set $E \cup \{e\}$, where $e$ is any element not in $E$, with rank function defined as follows for any set $F \subseteq E$: \vspace{0.5 em}
\begin{align*}
\rk_{(M^* + e)^*} F &= \left\{\begin{array}{ll}
\rk_M F + 1, & \text{if}\; \cl_{M^*}(E \setminus F) \neq E \\[0.5 em]
\rk_M F, & \text{if}\; \cl_{M^*}(E \setminus F) = E
\end{array}\right. \\[0.5 em]
\rk_{(M^* + e)^*}(F \cup e) &= \rk_M F + 1 \\[-0.5 em]
\end{align*}
Using the description of the rank function of the dual of $M$, it is easily checked that the condition $\cl_{M^*}(E \setminus F) = E$ is satisfied if and only if $\rk_M F = \abs{F}$, which is equivalent to $F$ being an independent set of $M$.  From this, it can be seen that a set of the form $F \cup e$ is a flat of the free coextension of $M$ if and only if $F$ is a flat of $M$, and a set $F$ is a flat of the free coextension if and only if $F$ is an independent set of $M$.  Thus, in a way, the free coextension treats independent sets and flats of $M$ on equal footing.  By identifying the flat $F$ of $M$ with the flat $F \cup e$ of the free coextension, we see that the lattice of  flats of the free coextension contains an isomorphic copy of the lattice of flats of $M$, just shifted upward in rank by one, in addition to all independent sets of $M$ with their usual rank.  

Consider the subset $\mathcal{G}_\aug \subseteq \LL((M^* + e)^*)$ defined by
\[
\mathcal{G}_\aug = \{ i \mid i \in E\} \cup \{F \cup e \mid F \in \LL(M) \}.
\]
We claim that $\mathcal{G}_\aug$ is a building set for the lattice of flats of the free coextension.  By the preceding paragraph, we know that every non-minimal element of the lattice of flats of the free coextensions is either a nonempty independent set of $M$ or of the form $F \cup e$ for some flat $F$ of $M$.  When $F$ is a flat of $M$, $F \cup e \in \mathcal{G}_\aug$ so that the required poset isomorphism for $\mathcal{G}_\aug$ to be building is just the identity map $[\emptyset, F \cup e] \to [\emptyset, F \cup e]$.  On the other hand, when $F$ is an independent set of $M$, $\max (\mathcal{G}_\aug)_{\leq F} = \{ i \mid i \in F\}$.  Since the interval $[\emptyset, F]$ is just the Boolean lattice of all subsets of $F$, it is easily seen that the map $\phi_F: \prod_{i \in F} [\emptyset, i] \to [\emptyset, F]$ given by $\phi_F(A_1, \dots, A_r) = \bigcup_{j = 1}^r A_j$ is the required poset isomorphism for $\mathcal{G}_\aug$ to be building.

The \term{Feichtner-Yuzvinsky presentation} of the augmented Chow ring of a simple matroid $M$ is the ring
\[ \aChow_{FY}(M) = \QQ\left[y_i, y_{F \cup e} \mid i \in E, F \in \LL \right]/I_{FY}(M), \]
where
\begin{align*}
I_{FY}(M) &= \left(y_{F \cup e}y_{F' \cup e} \mid F, F' \; \text{incomparable}\;\right) + (y_iy_{F \cup e} \mid i \in E, i \notin F) \\
&\qquad \qquad + (\textstyle y_i + \sum_{i \in F}\, y_{F \cup e}, \sum_F \, y_{F \cup e} \mid i \in E). 
\end{align*}
Setting $x_F = y_{F \cup e}$ and using the last linear form above to eliminate $x_E$ recovers the usual presentation of the augmented Chow ring of $M$.

\begin{defn}
The \term{atom-free presentation} of the augmented Chow ring of $M$ is the ring
\[ \aChow_\af(M) = \QQ[x_F \mid F \in \LL \setminus \{\emptyset\} ]/I_\af(M), \]
where
\begin{align*} 
I_\af(M) 
&= (x_Fx_{F'} \mid F, F' \; \text{incomparable}\;)  + (\textstyle x_F \sum_{i \in F'}\, x_{F'} \mid i \in E, i \notin F ) + ((\textstyle \sum_{i \in F'} \, x_{F'})^2 \mid i \in E)
\end{align*}
\end{defn}

After a change variables $\phi$ on the ring $\QQ[y_i, x_F \mid i \in E, F \in \LL]$ sending $y_i \mapsto y_i - \sum_{i \in F} x_F$ for each $i \in E$ and sending $x_\emptyset \mapsto x_\emptyset - \sum_{F \neq \emptyset} x_F$, we see that 
\begin{align} 
\begin{split} \label{augmented:change:of:variables}
\phi(I_{FY}(M)) &= \left(x_\emptyset, y_i \mid i \in E) + (x_Fx_{F'} \mid F, F' \in \LL \setminus \{\emptyset\} \; \text{incomparable}\;\right)  \\ 
&+ (\textstyle x_F\sum_{i \in F'} \, x_{F'} \mid F \in \LL \setminus \{\emptyset\}, i \in E, i \notin F ) \\
&+ (\textstyle (\sum_{i \in F'} \, x_{F'})(\sum_{F \neq \emptyset} \,x_F) \mid i \in E) 
\\
&= (x_\emptyset, y_i \mid i \in E) + I_\af(M),
\end{split}
\end{align}
where the last equality follows after a suitable change of generators from the fact that 
\[
\textstyle (\sum_{i \in F'} \, x_{F'})(\sum_{F \neq \emptyset} \,x_F) = (\sum_{i \in F'} \, x_{F'})^2 + \sum_{F \neq \emptyset, i \notin F} x_F(\sum_{i \in F'} \, x_{F'}). 
\] 
And so, we see that 
\[ 
\aChow(M) \iso \aChow_{FY}(M) \iso \frac{\QQ[y_i, x_F \mid i \in E, F \in \LL]}{(x_\emptyset, y_i \mid i \in E) + I_\af(M)} \iso \aChow_\af(M).
\]

\begin{thm}[{\cite[Theorem 2]{FY04}}]
With respect to any lexicographic order $>$ such that $x_F > x_G$ implies $F \nsupseteq G$ and $y_i > x_F$ for all $i \in E$ and all flats $F$, the defining ideal $I_{FY}(M)$ of the augmented Chow ring of $M$ has a Gr\"obner basis consisting of the following polynomials for all $i \in E$ and $F, F' \in \LL$:
\begin{align} 
&x_{F'}x_F & F, F' \text{incomparable} \\
& y_ix_F & i \notin F \\
&x_{F'}(\textstyle \sum_{G \supseteq F}\, x_G)^{\rk F - \rk F'}& F' \subsetneq F \\
&y_i(\textstyle \sum_{G \supseteq F} \, x_G)^{\rk F}& i \in F \\
&(\textstyle \sum_{G \supseteq F} \, x_G)^{\rk F + 1}
\\
& \textstyle y_i + \sum_{i \in F} \, x_F & i \in E
\end{align}
\end{thm}

A slight modification yields a Gr\"obner basis in the atom-free setting.  We omit the proof since it is completely analogous to the proof of Corollary \ref{atom-free:Gröbner:basis}.

\begin{cor} \label{augmented:atom-free:Gröbner:basis}
With respect to any lexicographic order $>$ such that $x_F > x_G$ implies $F \nsupseteq G$, the defining ideal $I_\af(M)$ of the augmented Chow ring of $M$ has a Gr\"obner basis consisting of the following polynomials for all $F, F' \in \LL \setminus \{\emptyset\}$:
\begin{align} 
&x_Fx_{F'}& F, F' \; \text{incomparable} \\
&x_{F'}(\textstyle \sum_{G \supseteq F} \, x_G)^{\rk F - \rk F'}& F' \subsetneq F \\
&(\textstyle \sum_{G \supseteq F} \, x_G)^{\rk F + 1} &
\end{align}
\end{cor}

\begin{cor} \label{augmented:monomial:basis}
The ring $\aChow_\af(M)$ has a $\QQ$-basis consisting of all monomials
\[ x_F^\alpha = x_{F_1}^{\alpha_1} \cdots x_{F_r}^{\alpha_r} \]
for all chains of flats $F = \{F_1 \supsetneq F_2 \supsetneq \cdots \supsetneq F_r \supsetneq F_{r+1} = \emptyset \}$ for some $r \geq 0$ and all $\alpha \in \ZZ_{\geq 0}^{r+1}$ such that $\sum_i \alpha_i = \rk F_1 + 1$, $1 \leq \alpha_i < \rk F_i - \rk F_{i + 1}$ if $i < r$, and $\alpha_r \leq \rk F_r$.
\end{cor}

\noindent As before, we call the monomials of Corollary \ref{augmented:monomial:basis} the \term{nested monomials} of $\aChow_{\mathrm{af}}(M)$.

\subsection{A Koszul filtration for the augmented Chow ring of a matroid}\label{SSfilt2}

When $M$ is a matroid of rank 1, we note that $\aChow(M) \iso \QQ[x_E]/(x_E^2)$, which we have previously noted has a Koszul filtration.  We may, therefore, assume the following without any loss of generality.

\begin{notation}
Throughout the remainder of this section, $M$ denotes a simple matroid with $\rk M \geq 2$ and $A = \aChow(M)$ with respect to the atom-free presentation.
\end{notation}

\begin{rmk} \label{augmented:socle:generator}
As a consequence of Corollary \ref{augmented:atom-free:Gröbner:basis}, we note that $I_{\mathrm{af}}(M)$ contains the polynomials
\begin{equation} \label{augmented:hyperplane:relations}
x_Hx_E, \qquad x_Fx_H^{\rk H - \rk F} + x_Fx_E^{\rk M - \rk F - 1}, \qquad x_H^{\rk H +1} + x_E^{\rk M}
\end{equation}
for any $F \in \LL \setminus \{\emptyset\}$ and hyperplane $H \supsetneq F$.  In particular, we note that the last type of relation is the primary difference when compared with the ordinary Chow ring of $M$.  Since $x_Fx_E^{\rk M} = 0$ in $A$ for all $F \in \LL \setminus \{\emptyset\}$ and $A$ is an Artinian Gorenstein ring, it follows that $x_E^{\rk M}$ is a socle generator of $A$.
\end{rmk}

The proofs of the following results are completely analogous to those of Section \ref{Smain} aside from some minor modifications based on the above remark.

\begin{lemma} \label{augmented:hyperplane:ideal:basis}
Let $\mathcal{H}$ be a set of hyperplanes of $M$.  Then the ideal $(x_H \mid H \in \mathcal{H})$ in $A = \aChow(M)$ has a $\QQ$-basis consisting of all nested monomials $x_F^\alpha$ with $F = \{F_1 \supsetneq F_2 \supsetneq \cdots \supsetneq F_r \supsetneq F_{r+1} = \emptyset \}$ for some $r \geq 0$ such that either: 
\begin{enumerate}[label = \textnormal{(\roman*)}]
\item $F_1 \in \mathcal{H}$, or 
\item $F_1 = E$, $H \supsetneq F_2 \neq \emptyset$ for some $H \in \mathcal{H}$ and $\alpha_1 = \rk M - \rk F_2 - 1$, or 
\item $F_1 = E$, $F_2 = \emptyset$, and $\alpha_1 = \rk M$.
\end{enumerate}
\end{lemma}

\begin{prop} \label{augmented:truncation}
For any matroid $M$, we have $(0 : x_E) = (x_H \mid H \in \coat(E))$ and $A/(0 : x_E) \iso \aChow(T(M))$. 
\end{prop}

\begin{cor} \label{augmented:colons:of:arbitrary:flats}
Let $F$ be a flat of $M$ with $\rk F \geq 1$.
\begin{enumerate}[label = \textnormal{(\alph*)}]
\item \label{augmented:kernel:of:restriction} The kernel of the natural surjective homomorphism $\pi_{M,F}: A \to \aChow(M \vert F)$ sending $x_G \mapsto x_G$ if $G \subseteq F$ and $x_G \mapsto 0$ otherwise is $(x_G \mid G \nsubseteq F)$.
\item \label{augmented:colon:with:filter} Suppose that $\mathcal{G}$ is a set of flats of $M$ such that $(F, E] \subseteq \mathcal{G}$ and $\mathcal{G} \cap [\emptyset, F] = \emptyset$. If $\rk F \geq 2$, then
\[ 
(x_G \mid G \in \mathcal{G}) : x_F = (x_G \mid G \nsubseteq F) + (x_G \mid G \in \coat(F)). 
\]
Otherwise, $(x_G \mid G \in \mathcal{G}) : x_F = A_+$.
\end{enumerate}
\end{cor}


\begin{prop} \label{augmented:colon:computations}
Let $\mathcal{H}$ be a nonempty set of hyperplanes of $M$ and $H'$ be any hyperplane.  Then:
\begin{enumerate}[label = \textnormal{(\alph*)}]

\item  \label{augmented:restriction}
$(0 : x_{H'}) = (x_F \mid F \nsubseteq H' )$ and $A/(0: x_{H'}) \iso \aChow(M\vert H')$.

\item \label{augmented:annihilator:of:hyperplane:ideal}
$(0 : (x_H \mid H \in \mathcal{H})) = (x_F \mid F \nsubseteq H \;\text{for all}\; H \in \mathcal{H})$.

\item \label{augmented:hyperplane:colons}
Suppose that $H' \notin \mathcal{H}$.  If $\rk H' \geq 2$ and for every $H \in \mathcal{H}$ there exists an $F \in \coat_\mathcal{H}(H')$ such that $H \cap H' \subseteq F$, then
\[
(x_H \mid H \in \mathcal{H}) : x_{H'} =  (0 : x_{H'}) + (x_F \mid F \in \coat_\mathcal{H}(H'))
\]
Otherwise, if $\rk H' = 1$, then $(x_H \mid H \in \mathcal{H}) : x_{H'} = A_+$. 

\end{enumerate}
\end{prop}

\begin{thm} \label{augmented:Chow:rings:of:matroids:are:Koszul}
The augmented Chow ring of a matroid has a Koszul filtration.  In particular, such rings have rational Poincar\'e series.
\end{thm}

It is tempting to seek a common framework for the Koszul property of Chow rings and augmented Chow rings of matroids.  The first natural approach is to consider atomic Chow rings of geometric lattices with respect to arbitrary building sets.  Examples~\ref{nonKoszul:building:set}--\ref{nonKoszulminimalbuildingset} show that there is no such result in this level of generality.

\section{Examples and questions}\label{Sexamples}

In this section, we collect some relevant examples and questions.  We begin with B{\o}gvad's algebra with irrational Poincar\'e series.

\begin{example}[\textbf{A quadratic Artinian Gorenstein algebra with irrational Poincar\'{e} series}]\label{EXirrational}

This example is due to B{\o}gvad \cite{Bog83} and is based on Anick's \cite{Ani82} construction;  however, it is not obvious that this construction leads to a quadratic algebra.   Let $\kk$ denote a field of characteristic $2$.  (Computation suggests that any field will do.)  Set
\[R = \kk[x_1,x_2,x_3,x_4,x_5]/(x_1^2,x_2^2,x_3^2,x_4^2,x_5^2,x_1x_2,x_4x_5,x_1x_3+x_3x_4+x_2x_5).\]
Anick proved that $R$ has irrational Poincar\'{e} series.  One can easily check that $R$ is superlevel in the language of Mastroeni, Schenck, and Stillman \cite{MSS21}.  Therefore, the idealization $\tilde{R} = R \ltimes \omega_R(-3)$ is quadratic and Gorenstein; here $\omega_R$ denotes the canonical module of $R$.  The Poincar\'{e} series of $\tilde{R}$ is rationally related to that of $R$ \cite{Gul72}, and so $\tilde{R}$ also has irrational Poincar\'{e} series.  (In both papers, the ideal $(x_1,x_2,x_3,x_4,x_5)^3$ is included among the relations but is redundant.)
\end{example}

One might hope for a more general result about the Koszul property of Chow rings of atomic lattices and arbitrary building sets.  Even for geometric lattices, the following examples show that this is not possible.

\begin{example}[\textbf{A geometric lattice with Gorenstein but not Koszul Chow ring}] \label{nonKoszul:building:set}
In general, the Chow rings $D(\LL, \mathcal{G})$ of Feichtner and Yuzvinsky need not be Koszul for all building sets $\mathcal{G}$ even when $\LL$ is a geometric lattice.  For example, suppose $\LL = B_3$ is the Boolean lattice on 3 elements, which is the lattice of flats of the uniform matroid $U_{3,3}$ on the ground set $\{1, 2, 3\}$.  Then $\mathcal{G} = \{\{1\}, \{2\}, \{3\}, \{1, 2, 3\}\}$ is easily checked to be a building set for $\LL$, but (after suppressing commas and braces in the variable indices) the Chow ring is
\[ 
D(\LL, \mathcal{G}) = \frac{\QQ[x_1, x_2, x_3, x_{123}]}{(x_1x_2x_3, x_1 + x_{123}, x_2 + x_{123}, x_3 + x_{123})} 
\iso
\frac{\QQ[x_{123}]}{(x_{123}^3)},
\]
which is evidently not Koszul.  Interestingly, one can check using Macaulay2 that this is the only building set for the lattice $B_3$ for which the Chow ring is not Koszul.
\end{example}

\begin{example}[\textbf{A geometric lattice with non-Koszul Chow ring for the minimal building set}] \label{nonKoszulminimalbuildingset}

Even when $\mathcal{G} = \mathcal{G}_{\min}$ is the minimal building set of geometric lattice $\LL$, $D(\LL,\mathcal{G})$ can fail to be Koszul.  Let $C_4$ denote the $4$-cycle and $M = M(C_4)$ its graphic matroid.  The minimal building set of $\LL(M)$ is $\mathcal{G} = \{\{1\},\{2\},\{3\},\{4\},\{1,2,3,4\}\}$.  The Chow ring is
\[ 
\frac{\QQ[x_1, x_2,x_3,x_4, x_{1234}]}{(x_1x_2x_3, x_1x_2x_4,x_1x_3x_4, x_2x_3x_4, x_1 + x_{1234}, x_2 + x_{1234}, x_3 + x_{1234},x_4 + x_{1234})} 
\iso
\frac{\QQ[x_{1234}]}{(x_{1234}^3)},
\]
which is also not Koszul.
\end{example}

The Gorenstein property of the (augmented) Chow ring of a matroid seems to play an important role in our methods.  It was already clear from the work of Feichtner and Yuzvinsky that not all Chow rings of atomic lattices were Gorenstein; see \cite[pp. 22-23]{FY04}.  We present a simple example.

\begin{example}[\textbf{A non-geometric lattice with Koszul but not Gorenstein Chow ring}]
 Consider the lattice $\LL$ with Hasse diagram in Figure~\ref{lat1}.  It is easy to check that $\LL$ is atomic but not semimodular.  The Chow ring of $\LL$  with respect to the maximal building set $\mathcal{G}_{\max} = \LL \setminus \{\hat{0}\}$ is
\[ 
D(\LL,\mathcal{G}_{\max}) \iso \frac{\QQ[x_e, x_f, x_{\hat{1}}]}{(x_e, x_f, x_{\hat{1}})^2},
\]
which is Koszul but not Gorenstein.

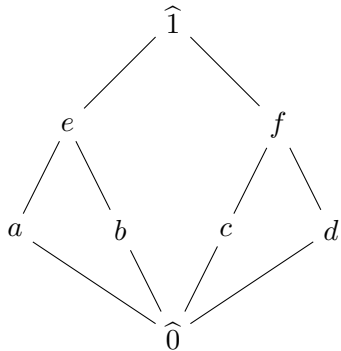
\begin{figure}[ht] 
\begin{center}
\begin{tikzpicture}[scale=.7]
  \node (e) at (-2,2) {$e$};
  \node (f) at (2,2) {$f$};
  \node (one) at (0,4) {$\hat1$};
  \node (a) at (-3,0) {$a$};
  \node (b) at (-1,0) {$b$};
  \node (c) at (1,0) {$c$};
  \node (d) at (3,0) {$d$};
  \node (zero) at (0,-2) {$\hat0$};
  \draw (zero) -- (a) -- (e) -- (one) -- (f) -- (d) -- (zero) ;
  \draw (e) -- (b) -- (zero) -- (c) -- (f);
\end{tikzpicture}
\caption{An atomic lattice $\LL$ with non-Koszul Chow ring $\Chow(\LL,\mathcal{G}_{\max})$.}\label{lat1}
\end{center}
\end{figure}

\end{example}

The preceding examples naturally lead to the following questions.

\begin{question}
Under what combinatorial conditions on the building set $\mathcal{G}$ and the lattice $\LL$ is the Chow ring $D(\LL, \mathcal{G})$ Koszul?  Is there a  Chow ring of a geometric lattice that fails to be Koszul with respect to some building set while still being quadratic?
\end{question}

\begin{question}
Is the Chow ring of any graded, atomic lattice with respect to its maximal building set always Koszul?
\end{question}

While Dotsenko showed that $D(\Pi_n,\mathcal{G}_{\mathrm{min}})$ has a quadratic Gr\"obner basis, a similar monomial order on $\Chow(M) = D(\LL(M),\mathcal{G}_{\mathrm{max}})$ does not seem to produce a quadratic Gr\"obner basis.  On the other hand, there does not appear to be a Hilbert function obstruction to one, leaving open the following question.

\begin{question}
Is the Chow ring of a matroid G-quadratic?
\end{question}

We close with a potential generalization of our results.

\begin{question}
Does every quadratic, Artinian, Gorenstein $\kk$-algebra with the K\"ahler package have the Koszul property?
\end{question}

It follows from \cite[Theorem 4.3]{MS20} that not all quadratic, Gorenstein $\kk$-algebras have the K\"ahler package; in particular, there are quadratic Gorenstein $\kk$-algebras with non-unimodal $h$-vectors.

\section*{Acknowledgements}
The authors thank Emanuele Delucchi, Vladimir Dotsenko, Chris Eur, Vic Reiner, and Botong Wang for many helpful conversations. 
Computations with Macaulay2 \cite{M2}, especially with Justin Chen's matroid package \cite{Che19}, were very helpful while working on this project.  Mastroeni was supported by an AMS-Simons Travel Grant.  McCullough was supported by National Science Foundation grant DMS--1900792. 
\end{spacing}

\bibliographystyle{alpha}
\bibliography{chow}

\end{document}